\documentclass[reqno,12pt,letterpaper]{amsart}
\usepackage{amsmath,amssymb,amsthm,graphicx,mathrsfs,url}
\usepackage[usenames,dvipsnames]{color}
\usepackage[colorlinks=true,linkcolor=Red,citecolor=Green]{hyperref}

\usepackage{latexsym}
\usepackage{eucal}
\usepackage{enumerate}

\def\?[#1]{\textbf{[#1]}\marginpar{\Large{\textbf{??}}}}
\renewcommand{\Re}{\mathop{\rm Re}\nolimits}
\renewcommand{\Im}{\mathop{\rm Im}\nolimits}

\DeclareMathOperator{\WF}{WF}
\newcommand{\WFh}{\WF_h}
\DeclareMathOperator{\tr}{tr}
\DeclareMathOperator{\el}{ell}
\newcommand{\elh}{\el_h}
\DeclareMathOperator{\rank}{rank}
\DeclareMathOperator{\supp}{supp}
\DeclareMathOperator{\Res}{Res}
\DeclareMathOperator{\neigh}{neigh}
\newcommand{\RR}{\mathbb R}
\newcommand{\CC}{\mathbb C}
\newcommand{\CIc}{{C^\infty_{\rm{c}}}}
\newcommand{\CI}{{C^\infty}}

\newtheorem{thm}{Theorem}
\newtheorem{prop}{Proposition}

\newtheorem{lem}[prop]{Lemma}

\numberwithin{equation}{section}
\numberwithin{prop}{section}

\makeatletter
\let\@maketoccontribs\maketoccontribs
\makeatother

\begin{document}
\title{A local trace formula for Anosov flows}
\author{Long Jin}
\address{CMSA, Harvard University, Cambridge, MA 02138, USA}
\email{jinlong@cmsa.fas.harvard.edu}
\author[M. Zworski]{Maciej Zworski}
\address{Department of Mathematics, University of California, Berkeley,
CA 94720, USA}
\email{zworski@math.berkeley.edu}
\address{Laboratoire d'Analyse non
lin\'eaire et g\'eom\'etrie, Universit\'e d'Avignon,
Avignon, France}
\email{frederic.naud@univ-avignon.fr}
\date{}

\maketitle

\vspace{-0.2in}

\begin{center}
{\sc With appendices by Fr\'ed\'eric Naud}
\end{center}

\vspace{-0.1in}

\begin{abstract}
We prove a local trace formula for Anosov flows. It relates
Pollicott--Ruelle resonances to the periods of closed orbits.
As an application, we show that the counting function for resonances in a sufficiently wide strip cannot have a sublinear growth. In particular, for any Anosov flow there exist strips with infinitely many resonances.
\end{abstract}

\section{Introduction}
Suppose $X$ is a smooth
compact manifold and $\varphi_t:X\to X$ is an Anosov flow
generated by a smooth vector field $ V $, $ \varphi_t := \exp t V $.
Correlation functions for a flow are defined as
\begin{equation}
\label{eq:corr}
\rho_{f,g} ( t ) := \int_X f ( \varphi_{-t} ( x ) ) g ( x ) dx , \ \
f , g \in C^\infty ( X ) , \ \ t > 0 ,
\end{equation}
where $ dx $ is a Lebesgue density on $ X $.
The power spectrum is defined as the (inverse) Fourier-Laplace transform
of $ \rho_{f,g} $:
\begin{equation}
\label{eq:powersp}
\widehat \rho_{f,g} ( \lambda ) := \int_0^\infty \rho_{f,g} ( t ) e^{ i \lambda t } dt,
\ \ \ \Im \lambda > 0 .
\end{equation}
Faure--Sj\"ostrand \cite{FaSj} proved that
\[   ( P - \lambda)^{-1}  : C^\infty ( X )
\to {\mathcal D}' ( X ) , \ \ P := \frac 1 i V , \ \ \Im \lambda \gg 1 , \]
continues to a meromorphic family of operators on all of $ \mathbb C $. Using the fact
that $ f ( \varphi_{-t} ( x ) ) = [\exp ( - i t P ) f ] ( x ) $ this easily shows
that $ \widehat \rho_{f,g} ( \lambda ) $ has a meromorphic continuation.
The poles of this continuation depend only on $ P $ and their study
was initiated in the work of Ruelle \cite{Rue} and Pollicott \cite{Po}.
They are called {\em Pollicott--Ruelle resonances} and their set is denoted
by $ \Res ( P )$. The finer properties of the correlations are then
related to the distribution of these resonances. This is particularly
clear in the work of Liverani \cite{Liv} and Tsujii \cite{Ts} on contact
Anosov flows, see also Nonnenmacher--Zworski \cite{NZ} for semiclassical generalizations.

An equivalent definition of Pollicott--Ruelle resonances was given
by Dyatlov--Zworski \cite{DZ2}: they are limits (with multiplicities)
of the eigenvalues of $ P + i \epsilon \Delta_g $, $ \epsilon \to 0 + $,
where $ - \Delta_g \geq 0 $ is a Laplacian for some { Riemannian metric $g$} on $ X $.
Because of a connection to Brownian motion this shows stochastic stability of these resonances.

In this note we address the basic question about the size of the set of
resonances: is their number always infinite? Despite the long
tradition of the subject this appeared to be unknown for arbitrary
Anosov flows on compact manifolds.
In Theorem \ref{thm2}, we show that
in sufficiently large strips the counting function of resonances cannot
be bounded by $r^\delta$, $ \delta < 1 $.

\begin{figure}
\includegraphics[width=6in]{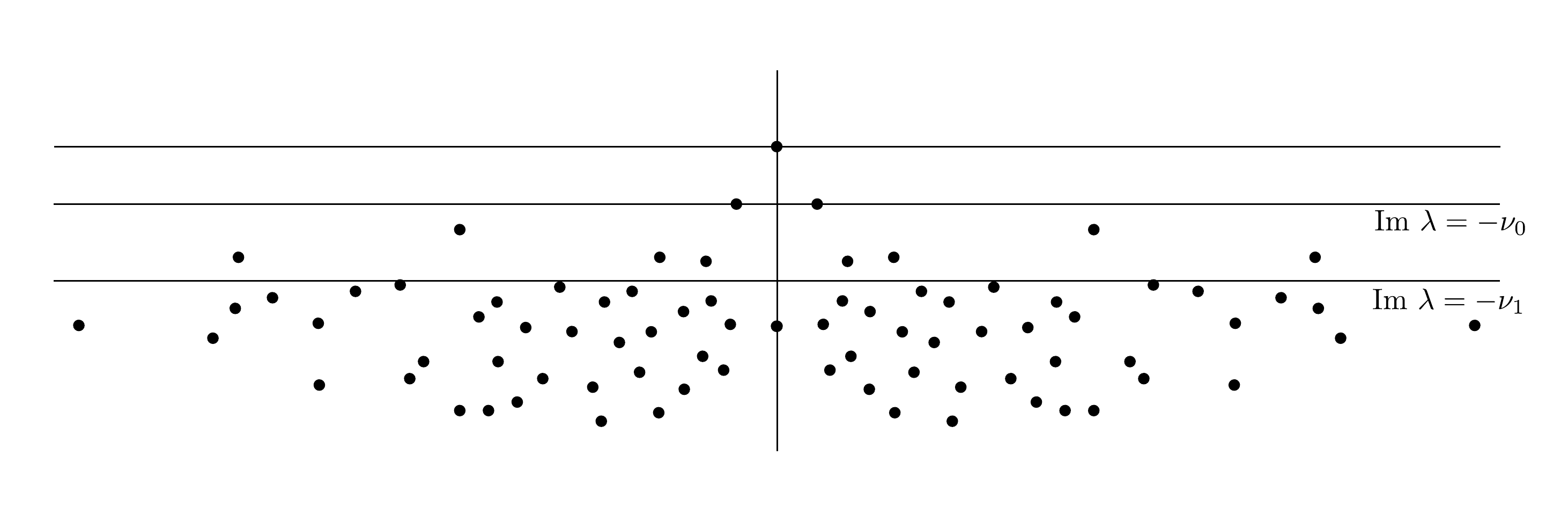}
\caption{The {\em spectral gap} $ \nu_0 $ is the supremum of $ \nu $
such that there are {\em no} resonances with $ - \nu < \Im \lambda $,
$ \lambda \neq 0 $.
For contact Anosov flows it is known that $ \nu_0 > 0 $ \cite{Liv},\cite{NZ},\cite{Ts}.
The {\em essential spectral gap}, $  \nu_1 $, is the supremum of $ \nu $
such that there are only finitely many resonances with $ \Im \lambda > - \nu $.
Our result states that the essential spectral gap is finite for any
Anosov flow on a compact manifold.}
\end{figure}

{General upper bounds
on the number of resonances in strips were established by
Faure--Sj\"ostrand \cite{FaSj} (and with a sharper 
exponent in the case of contact flows by Datchev--Dyatlov--Zworski \cite{DDZ}):
for any $A > 0 $ there exists $ C $ such that
\begin{equation}
\label{eq:counting}
\#(\Res(P)\cap\{\Im\mu>-A, |\Re\mu - r |\leq \sqrt r\})
\leq Cr^{n- \frac12} .
\end{equation}}
On the other hand, for contact Anosov flows satisfying certain pinching conditions on Lyapunov exponents, Faure--Tsujii \cite{FaTs} showed that
the resonances satisfy a precise counting law in strips, agreeing with
the upper bound of \cite{DDZ}. That is a far reaching generalization of the results known in constant curvature: see Dyatlov--Faure--Guillarmou \cite{DFG} for recent results in that case and references.

The new counting result is proved by
establishing a local trace formula relating resonances to periods of
closed trajectories and to the their Poincar\'e maps. Hence we denote
by  $\mathcal{G}$ periodic orbits $\gamma$ of the flow, by  $T_\gamma$ the period of $\gamma$ and by $T_\gamma^\#$ the primitive period. We let $\mathcal{P}_\gamma$ be the linearized Poincar\'{e} map -- see \S \ref{pr}.
With this notation we can state our {\em local trace formula}:
\begin{thm}
\label{thm1}
For any $A>0$ there exists
a distribution $F_A\in\mathcal {S}'(\mathbb{R})$ supported in $[0,\infty)$ such that
\begin{equation}
\label{localtrace}
\sum_{\mu\in\Res(P), \Im\mu>-A}e^{-i\mu t}+F_A(t)
=\sum_{\gamma\in\mathcal{G}}\frac{T_\gamma^{\#}\delta(t-T_\gamma)}
{|\det(I-\mathcal{P}_\gamma)|},\;\;\;\; t>0
\end{equation}
in the sense of distribution on $(0,\infty)$.
Moreover, the Fourier transform of $ F_A $ has an analytic extension
to $\Im\lambda <  A$ which satisfies,
\begin{equation}
\label{error}
|\widehat{F}_A(\lambda) |=\mathcal{O}_{A,\epsilon}( \langle \lambda \rangle^{2n+1}), \ \
\Im\lambda < A - \epsilon, \text{ for any $\epsilon>0$.}
\end{equation}
\end{thm}

The trace formula \eqref{localtrace} can be motivated as follows.
For the case of geodesic flows of compact Riemann surfaces,
$X=S^*(\Gamma\backslash\mathbb{H}^2)$, where $ \Gamma $ is
co-compact subgroup of $ {\rm{SL}}_2 ( \mathbb R) $,
and $\varphi_t$ is the geodesic flow, we have a global trace formula:
\begin{equation}
\label{globaltrace}
\sum_{\mu\in\Res(P)}e^{-i\mu t}=\sum_{\gamma\in\mathcal{G}}\frac{T_\gamma^\#\delta(t-T_\gamma)}{|\det(I-\mathcal{P}_\gamma)|},\;\;\; t>0.
\end{equation}
Here the set of resonances is given by
$\Res(P)=\{\mu_{j,k}=\lambda_j-i(k+\frac{1}{2}),j,k\in\mathbb{N}\}$ (up to exceptional values on the imaginary axis),
where $\lambda_j$'s are the eigenvalues
of the Laplacian on $ \Gamma \backslash \mathbb H^2 $.  This follows
from the Atiyah--Bott--Guillemin trace formula and
the Selberg trace formula -- see \cite{DFG} and references given there.
{  The bound $ \langle \lambda \rangle^{ 2 n +1 } $ in 
\eqref{error} is probably not
optimal and comes from very general estimates presented in \S \ref{estflt}.
It is possible that \eqref{globaltrace} is valid for all Anosov flows.}

Melrose's Poisson formula for
resonances valid for Euclidean infinities \cite{Me,SZ,Z1} and some
hyperbolic infinities \cite{GZ} suggests that \eqref{globaltrace}
could be valid for general Anosov flows but that is unknown.

In general, the validity of \eqref{globaltrace} follows from, but is not equivalent to, the finite order (as an entire function) of the analytic continuation of
\begin{equation}
\label{zeta}
\zeta_1(\lambda) { :=} \exp\left(-\sum_\gamma\frac{T_\gamma^\#e^{i\lambda T_\gamma}}{T_\gamma|\det(I-\mathcal{P}_\gamma)|}\right).
\end{equation}
This finite order property is only known under certain analyticity assumptions on $X$ and $\varphi_t$ -- see Rugh \cite{Ru} and Fried \cite{Fr}. {  
The notation $ \zeta_1 $ is motivated by the factorization of the Ruelle
zeta function -- see \cite[(2.5)]{DZ}.}

As a consequence of the local trace formula \eqref{localtrace},
we have the following weak lower bound on the number of resonances in a sufficiently wide strip near the real axis. It is formulated using the Hardy-Littlewood notation: $f=\Omega(g)$ if it is
not true that $|f|=o(|g|)$.

\begin{thm}
\label{thm2}
For every $\delta\in(0,1)$ there exists a constant $A_\delta>0$ such that if $A>A_\delta$, then
\begin{equation}
\label{lowerbound}
\#(\Res(P)\cap\{\mu \in \CC \, : \, |\mu|\leq r, \ \Im\mu>-A\})=\Omega(r^\delta).
\end{equation}
In particular, there are infinitely many resonances in any strip $\Im\mu>-A$ for $A$ sufficiently large. 
\end{thm}

\medskip
\noindent{\bf Remarks.} 1. { An explicit bound for the constant $A_\delta$ is given by \eqref{adelta} in the proof. This also gives an explicit bound $A_0=\inf\{A_\delta:0<\delta<1\}$ for the essential spectral gap.} In the case of analytic semiflows (see \cite{Naud1}) Fr\'ederic Naud \cite{Naud} pointed out that a better estimate of the essential spectral gap is possible: there are infinitely many resonances in any strip
$ \Im \lambda > - \frac32 P ( 2 ) - \epsilon $, where 
{  $ P  (s ) := P ( s \psi^u ) $ is the
topological pressure associated to the unstable Jacobian -- see \eqref{eq:press} and \eqref{eq:psiu}}.
In Appendix \ref{weakmix}, Fr\'ed\'eric Naud shows how similar methods and Theorem \ref{thm1} give
a narrower strip with infinitely many resonances for weakly mixing
Anosov flows.

\noindent
2. In the case of flows obtained by suspending Anosov maps the 
growth of the number of resonances in strips is linear -- see
Appendix \ref{suspe} by Fr\'ederic Naud 
for a detailed discussion of analytic perturbations of linear 
maps. That means that the exponent $ \delta $ close to one is
close to be optimal in general.

\medskip

The proof of Theorem \ref{thm1} uses the microlocal approach to
Anosov dynamics due to Faure--Sj\"ostrand \cite{FaSj} and
Dyatlov--Zworski \cite{DZ}. In particular we use the fact that
\begin{equation*}
\frac{d}{d\lambda} \log \zeta_1 ( \lambda ) = \tr^\flat e^{ i \lambda t_0 } \varphi_{-t_0}^* ( P - \lambda)^{-1} ,
\end{equation*}
and that the right hand side continues meromorphically with poles
with integral residues. Here the flat trace, $ \tr^\flat $,  is defined using a formal integration over the diagonal,  see \S \ref{flat},  with the justification provided by the crucial wave front set relation, see \S \ref{wavefront}. Some of the techniques are also related to the proof of Sj\"{o}strand's local trace formula for scattering resonances in the semiclassical limit \cite{S}. It is possible that an alternative
proof of Theorem \ref{thm1} could be obtained using the methods of Giulietti--Liverani--Pollicott \cite{GLP} employed in their proof of Smale's conjecture about zeta function (\cite{DZ} provided a simple microlocal proof of that conjecture).

The proof of Theorem \ref{thm2} is based on the proof of a similar
result in Guillop\'{e}--Zworski \cite{GZ} which in turn was inspired
by the work of Ikawa \cite{Ik} on existence of resonances in
scattering by several convex bodies.

\def\smallsection#1{\smallskip\noindent\textbf{#1}.}

\smallsection{Acknowledgements}
We would like to thank Semyon Dyatlov for helpful discussions and in particular for suggesting the decomposition \eqref{decomposition} which simplified the wave front arguments. We are also grateful to Fr\'ed\'eric Naud for sharing his
unpublished work \cite{Naud} with us and to the anonymous referee for
useful suggestions. This material is based
upon work supported by
the National Science Foundation under the grant and DMS-1201417.

\smallsection{Notation} We use the following notational
conventions: $ \langle x \rangle := ( 1 + |x|^2 )^{\frac12} $,
$ \langle u , \varphi \rangle$, for the the distributional
pairing of $ u \in \mathcal D' ( X ) $ (distributions on a compact
manifold $ X $), and $ \langle u , v \rangle_{ H}$ for
the Hilbert space inner product on $ H $.
We write $ f =  \mathcal O_\ell ( g)_B $ to mean that
$ \|f \|_B  \leq C_\ell  g $ where the norm (or any seminorm) is in the
space $ B$, and the constant $ C_\ell  $ depends on $ \ell $. When either $ \ell $ or $ B $ are absent then the constant is universal or the estimate is scalar, respectively. When $ G = \mathcal O_\ell ( g )_{B_1\to B_2 } $ then the operator $ B : H_1  \to H_2 $ has its norm bounded by $ C_\ell g $. By $ \neigh_U ( \rho ) $
we mean a (small) neighbourhood of $ \rho$ in the space $ U$.
We refer to \cite{DZ} and \cite{Z2} for the notational conventions from microlocal/semiclassical analysis as they appear in the text.

\section{Preliminaries}
\label{pr}
\subsection{Anosov flows}
Let $X$ be a compact Riemannian manifold, $V\in C^\infty(X;TX)$ be
a smooth non vanishing vector field and
and $\varphi_t=\exp tV:X\to X$ the corresponding flow.

The flow is called an {\em Anosov flow}
if the tangent space to $X$ has a continuous decomposition
$T_xX=E_0(x)\oplus E_s(x)\oplus E_u(x)$ which is invariant under the flow: $d\varphi_t(x)E_\bullet(X)=E_\bullet(\varphi_t(X))$,
$\bullet=s,u$, $E_0(x)=\mathbb{R}V(x)$,
and satisfies
\begin{equation*}
\begin{split}
|d\varphi_t(x)v|_{\varphi_t(x)}\leq Ce^{-\theta|t|}|v|_x,&\;\;\; v\in E_u(x), \ \ t<0\\
|d\varphi_t(x)v|_{\varphi_t(x)}\leq Ce^{-\theta|t|}|v|_x,&\;\;\; v\in E_s(x),
\ \ t>0,
\end{split}
\end{equation*}
for some fixed $C$ and $\theta>0$.

\subsection{Anisotropic Sobolev spaces}
Let us put $P=-iV:C^\infty(X)\to C^\infty(X)$;
then the principal symbol of $ P $, $p\in S^1(T^*X)$ (see
\cite[\S 18.1]{H3} or \cite[\S 14.2]{Z2} for this standard notation; an
overview of semiclassical and microlocal preliminaries needed in this
paper can be found in \cite[\S 2.3]{DZ}) is given by
$p(x,\xi)=\xi(V(x))$ which is homogeneous of degree 1.
The Hamilton flow of $ p $ is
the symplectic lift of $ \varphi_t $ to the cotangent bundle:
$e^{tH_p}(x,\xi)=(\varphi_t(x),({}^Td\varphi_t(x))^{-1}\xi)$.
We can define the dual decomposition $T^\ast_xX=E^\ast_0(x)\oplus E^\ast_s(x)\oplus E^\ast_u(x)$ where $E^\ast_0(x),E^\ast_s(x),E^\ast_u(x)$ are dual to $E_0(x),E_u(x),E_s(x)$, respectively. Then
\begin{equation*}
\begin{split}
\xi\not\in E^\ast_0(x)\oplus E^\ast_s(x)\Rightarrow d(\kappa(e^{tH_p}(x,\xi)),\kappa(E^\ast_u))\to0 \text{ as } t\to+\infty\\
\xi\not\in E^\ast_0(x)\oplus E^\ast_u(x)\Rightarrow d(\kappa(e^{tH_p}(x,\xi)),\kappa(E^\ast_s))\to0 \text{ as } t\to-\infty.
\end{split}
\end{equation*}
Here $\kappa:T^\ast X\setminus0\to S^\ast X := T^*X / \mathbb R_+ $
is the natural projection.

A microlocal version of anisotropic Sobolev spaces of
Blank--Keller--Liverani~\cite{BKL}, Baladi--Tsujii~\cite{BT}
and other authors was provided by Faure-Sj\"{o}strand \cite{FaSj}.
Here we used a simplified version from Dyatlov-Zworski \cite{DZ}. For that
we construct a function $m_G\in C^\infty(T^\ast X\setminus0;[-1,1])$ which is homogeneous of degree 0, is supported in a small neighbourhood of $E_s^\ast\cup E_u^\ast$ and satisfies
\begin{equation*}
m_G=1 \text{ near } E_s^\ast;\;\;\  m_G=-1 \text{ near } E_u^\ast;\;\;\ H_pm_G\leq0 \text{ everywhere. }
\end{equation*}
Next, we choose a pseudodifferential operator
$ G\in\Psi^{0+}(X)$, {  $ \sigma(G)=m_G(x,\xi)\log\langle\xi\rangle $}.
Then for $ s > 0 $, $\exp(\pm sG)\in\Psi^{s+}(X)$ -- see \cite[\S 8.2]{Z2}. The anisotropic Sobolev spaces are defined as
\begin{equation*}
H_{sG}:=\exp(-sG)L^2(X), \ \ \ \|u\|_{H_{sG}}: =\|\exp(sG)u\|_{L^2}.
\end{equation*}
By the construction of $G$, we have $H^s\subset H_{sG}\subset H^{-s}$.

\subsection{Properties of Resolvent}
We quote the following results about the resolvent of $P$, see \cite[Propositions 3.1, 3.2]{DZ}:
\begin{lem}
Fix a constant $C_0>0$. Then for $s>0$ large enough depending on $C_0$, $P-\lambda:D_{sG}\to H_{sG}$ is a Fredholm operator of index 0 in the region $\{\Im\lambda>-C_0\}$. Here the domain $D_{sG}$ of $P$ is the set of $u\in H_{sG}$ such that $Pu$ (in the distribution sense) is in $H_{sG}$ and it is a Hilbert space with norm $\|u\|_{D_{sG}}^2=\|u\|_{H_{sG}}^2+\|Pu\|_{H_{sG}}^2$.
\end{lem}

\begin{lem}
Let $s>0$ be fixed as above. Then there exists a constant $C_1$ depending on $s$, such that for $\Im\lambda>C_1$, the operator $P-\lambda:D_{sG}\to H_{sG}$ is invertible and
\begin{equation}
\label{resolvent}
(P-\lambda)^{-1}=i\int_0^\infty e^{i\lambda t}\varphi_{-t}^\ast dt,
\end{equation}
where $\varphi_{-t}^\ast=e^{-itP}$ is the pull back operator by $\varphi_t$. The integral converges in operator norm $H^s\to H^s$ and $H^{-s}\to H^{-s}$.
\end{lem}

The analytic Fredholm theory now shows that the resolvent $ \lambda
\mapsto R(\lambda)=(P-\lambda)^{-1}:H_{sG}\to H_{sG}$ forms a meromorphic family of operators with poles of finite rank. In the region $\Im\lambda>-C_0$, the Ruelle-Pollicott resonances are defined as the poles of $R(\lambda)$. They can be described as the meromorphic continuation of the Schwartz kernel of the operator on the right-hand side, thus are independent of the choice of $s$ and the weight $G$.
The mapping properties of $ ( P - \lambda)^{-1}$ and formula \eqref{resolvent}
show that the power spectrum \eqref{eq:powersp} has a meromorphic
continuation with the same poles. We note here that our definition
\eqref{eq:powersp} is different from the definition in \cite{Rue} but
the formula there can be expressed in terms of \eqref{eq:powersp}.

{We recall the following general upper bounds on the number of resonances from 
Faure--Sj\"ostrand \cite{FaSj}:
\begin{prop}
Let $\Res(P)$ be the set of Ruelle-Pollicott resonances. Then for any $C_0>0$,
\begin{equation}
\label{hupperbound}
\#(h\Res(P))\cap D(1,C_0 h^{\frac12} )=\mathcal{O}(h^{-n+ \frac12}),
\end{equation}
which is equivalent to \eqref{eq:counting}. In particular,
\begin{equation}
\label{upperbound}
\#\Res(P)\cap\{\mu:|\Re\mu|\leq r,\Im\mu>-C_0\}=\mathcal{O}(r^n). 
\end{equation}
\end{prop}}

\subsection{Complex absorbing potentials}
It is convenient to introduce a semiclassical parameter $h$ and to consider the operator $hP\in\Psi_h^1(X)$ (for the definitions of pseudodifferential
operators and wave front sets we
refer to \cite[\S 14.2]{Z2} and  \cite[\S 2.3, Appendix C]{DZ})
 with semiclassical principal symbol $p=\sigma_h(hP)(x,\xi)=\xi(V_x)$.
Then we introduce a semiclassical adaption $G(h)\in\Psi_h^{0+}(X)$ of the operator $G$ with
\begin{equation*}
\sigma_h(G(h))=(1-\chi(x,\xi))m_G(x,\xi)\log|\xi|,
\end{equation*}
where $\chi\in C_0^\infty(T^\ast X)$ is equal to 1 near the zero section. In this way, $H_{sG(h)}=H_{sG}$ but with a new norm depending on $ h $. We also
define an $h$-dependent norm on the domain of $hP$,  $D_{sG(h)} = D_{s G} $:
\begin{equation*}
\|u\|_{D_{sG(h)}} :=\|u\|_{H_{sG(h)}}+\|hPu\|_{H_{sG(h)}}.
\end{equation*}

Now we modify $hP$ by adding a semiclassical pseudodifferential complex absorbing potential $-iQ_\delta\in\Psi_h^0(X)$ which is localized to a neighbourhood of the zero section:
\begin{equation}
\label{eq:WFQd}
\WF_h(Q_\delta)\subset\{|\xi|<\delta\}, \ \ \sigma_h(Q_\delta)>0 \text{ on } \{|\xi|\leq\delta/2\}, \ \ \sigma_h(Q_\delta)\geqslant0 \text{ everywhere}.
\end{equation}
{  (For the definition of $ \WFh ( A ) \subset \overline T^* X $ and
of the compactified cotangent bundle $ \overline T^* X $, 
see \cite[\S C.2]{DZ}.)}
Instead of $P_h(z)=hP-z$, we consider the operator $P^\delta_h(z)=hP-iQ_\delta-z$ acting on $H_{sG(h)}$ which is equivalent to the conjugated operator
\begin{equation}
\label{eq:conj}
P^{\delta,s}_h(z)=e^{sG(h)}P_\delta(z)e^{-sG(h)}=P^\delta_h (z)+s[G(h),hP]+\mathcal{O}(h^2)_{\Psi_h^{-1+}}
\end{equation}
acting on $L^2$. We recall the crucial \cite[Proposition 3.4]{DZ}:
\begin{lem}
\label{l:3.4}
Fix a constant $C_0>0$ and $\delta >0$. Then for $s>0$ large enough depending on $C_0$ and $h$ small enough, the operator
\begin{equation*}
P^\delta_h(z):D_{sG(h)}\to H_{sG(h)}, \ \ \ -C_0h\leq\Im z\leq1, \ \
 |\Re z|\leq 2h^{1/2},
\end{equation*}
is invertible, and the inverse $R^\delta_h(z)$, satisfies
$\|R^\delta_h(z)\|_{H_{sG(h)}\to H_{sG(h)}}\leq Ch^{-1}$.
\end{lem}

\subsection{Finite rank approximation}
\label{fra}
For our application we need to make $ Q_\delta $ a finite rank
operator.
It is also convenient to make a further assumption on the
symbol of $ Q_\delta$.
As long as \eqref{eq:WFQd} holds,
Lemma \ref{l:3.4} still applies.

From now on we fix $ \delta > 0 $ and put
\[  Q = Q_\delta = f ( - h^2 \Delta_g ), \ \ f \in  \CIc ( ( -2 \delta, 2 \delta ),
[ 0 , 1 ]) , \ \
f ( s ) = 1, \ \ |s |\leq \delta . \]
Then (see for instance \cite[Theorem 14.9]{Z2})
\begin{equation}
\label{finiterank}
\rank Q=\mathcal O (h^{-n}),  \ \ Q \geq 0 , \ \ \sigma_h ( Q ) = f ( |\xi|_g^2 ).
\end{equation}
For technical convenience only (so that we can
cite easily available results in the proof of Proposition \ref{flattracees}
in the appendix)
we make an additional assumption on $ f$: for some $ 0 < \alpha < \frac12$,
\begin{equation}
\label{eq:condfk}
| f^{(k)} ( x ) | \leq C_k f ( x ) ^{1- \alpha} .
\end{equation}
This can be achieved by building $ f$ from functions of the form equal to $
e^{-1/x} $ for $ x > 0 $ and $ 0 $ for $ x \leq 0 $. (In that case
\eqref{eq:condfk} holds for all $ \alpha > 0 $.)

Lemma \ref{l:3.4} shows that
for  $-C_0h\leq\Im z\leq 1
$, $|\Re z|\leq 2h^{1/2}$, 
\begin{equation}
\label{eq:wideP} 
\widetilde{P}_h(z): =hP-iQ-z, 
\end{equation} 
is also invertible and its inverse $\widetilde{R}_h(z)$ satisfies
\begin{equation}
\label{modifiedresolvent}
\|\widetilde{R}_h(z)\|_{H_{sG(h)}\to H_{sG(h)}}\leq Ch^{-1}.
\end{equation}
In the upper half plane we have the following estimate on the original resolvent:
\begin{equation}
\label{eq:orres}
 \|R_h(z)\|_{H_{sG(h)}\to H_{sG(h)}}\leq Ch^{-1}, \ \ \
C_1h\leq\Im z\leq1, \ \ |\Re z|\leq 2h^{1/2},
\end{equation}
provided that $ C_1 $ is large enough. This follows from the
Fredholm property and the estimate $ \Im \langle e^{ s G ( h ) } P_h ( z )
e^{ -s G ( h ) } u , u \rangle_{L^2 }
\geq h \| u\|_{L^2} $, $ \Im z > C_1 h $ -- see \eqref{eq:conj}.

\subsection{Wavefront set condition}
\label{wavefront}
We need to study the wavefront set and semiclassical wavefront set of $R_h(z)$ and $\widetilde{R}_h(z)$. For the definitions and notations of the wavefront sets and the semiclassical wavefront sets, we refer to \cite[Chapter VIII]{H}, \cite[Section 8.4]{Z2} and \cite[Appendix C]{DZ} and \cite{A}.

We recall the following wavefront set condition and semiclassical wavefront set conditions for the resolvent $R(\lambda)$ and $\widetilde{R}_h(z)$ from \cite[Proposition 3.3]{DZ}. {  (For the definition of the standard wave front set $ \WF $
see \cite[\S C.1]{DZ} and for the definition of the twisted 
wave front set $ \WF' $, \cite[(C.2)]{DZ} -- the reason for the twist
is to have $ \WF ( I ) $ equal to the diagonal in $ T^*X \times T^* X $.)}
\begin{prop}
\label{p:3.3}
Let $C_0$ and $s$ be as above and assume $\lambda$ is not a resonance with $\Im\lambda>-C_0$, then
\begin{equation}
\label{wavefrontset}
\WF'(R(\lambda))\subset\Delta(T^\ast X)\cup\Omega_+\cup(E_u^\ast\times E_s^\ast),
\end{equation}
where $\Delta(T^\ast X)$ is the diagonal in $T^\ast X$ and $\Omega_+$ is the positive flow-out of $e^{tH_p}$ on $\{p=0\}$:
\begin{equation}
\label{eq:Omegapl}
\Omega_+=\{(e^{tH_p}(x,\xi),x,\xi) \, : \, t\geqslant0, \ p(x,\xi)=0\}.
\end{equation}
Also, if $R_h(z)=h^{-1}R(z/h)$, then
\begin{equation}
\label{semiwf1}
\WFh'(R_h(z))\cap T^\ast(X\times X)\subset\Delta(T^\ast X)\cup\Omega_+\cup(E_u^\ast\times E_s^\ast),
\end{equation}
and
\begin{equation}
\WFh'(R_h(z))\cap S^\ast(X\times X)\subset\kappa(\Delta(T^\ast X)\cup\Omega_+\cup(E_u^\ast\times E_s^\ast)\setminus\{0\}).
\end{equation}
\end{prop}

Now we determine the wavefront set and the semiclassical wavefront set of $\widetilde{R}_h(z)$. First, 
{  by inserting the resolvent formula
$$\widetilde{R}_h(z)=R_h(z)+iR_h(z)Q\tilde{R}_h(z)$$
into another resolvent formula
$$\widetilde{R}_h(z)=R_h(z)+i\tilde{R}_h(z)QR_h(z),$$}
we write
$$\widetilde{R}_h(z) = R_h ( z ) + i R_h ( z ) Q R_h ( z ) -
R_h ( z ) Q \widetilde{R}_h ( z ) Q R_h ( z ).$$
Then since $Q$ is a smoothing operator, $\WF(Q)=\emptyset$, we have
$$\WF'(R_h(z)QR_h(z))\subset E_u^\ast\times E_s^\ast.$$
Similarly, since $Q\widetilde{R}_h(z)Q$ is also a smoothing operator,
$$\WF'(R_h(z)Q\widetilde{R}_h(z)QR_h(z))\subset E_u^\ast\times E_s^\ast.$$
Therefore we get the same wavefront set condition as $R_h(z)$:
\begin{equation}
\WF'(\widetilde{R}_h(z))\subset\Delta(T^\ast X)\cup\Omega_+\cup(E_u^\ast\times E_s^\ast).
\end{equation}

For the semiclassical wavefront set, we already know from \cite[Proposition 3.4]{DZ} that
\begin{equation}
\label{eq:WFRh} \WFh'(\widetilde{R}_h(z))\cap T^\ast(X\times X)\subset\Delta(T^\ast X)\cup\Omega_+.
\end{equation}
Moreover, since $\WFh'(Q)\cap S^\ast(X\times X)=\emptyset$, we have
$$\WFh'(R_h(z)QR_h(z))\subset E_u^\ast\times E_s^\ast,$$
and similarly, $\WFh'(Q\widetilde{R}_h(z)Q)\cap S^\ast(X\times X)=\emptyset$. Therefore
\begin{equation}
\label{eq:WFhR}
\WFh'(\widetilde{R}_h(z))\cap S^\ast(X\times X)\subset
\kappa(\Delta(T^\ast X)\cup\Omega_+\cup(E_u^\ast\times E_s^\ast)\setminus\{0\}).
\end{equation}

\subsection{Flat trace}
\label{flat}
Consider an operator $B:C^\infty(X)\to\mathcal{D}'(X)$ with
\begin{equation}\label{ftc}
\WF'(B)\cap\Delta(T^\ast X)=\emptyset.
\end{equation}
Then we can define the flat trace of $B$ as
\begin{equation}
\label{eq:flat}
\tr^\flat B=\int_X(\iota^\ast K_B)(x)dx:=\langle\iota^\ast K_B,1\rangle
\end{equation}
where $\iota:x\mapsto (x,x)$ is the diagonal map, $K_B$ is the Schwartz kernel of $X$ with respect to the density $dx$ on $X$. The pull back $\iota^\ast K_B\in\mathcal{D}'(X)$ is well-defined under the condition \eqref{ftc} (see \cite[Section 8.2]{H}).

\subsection{Dynamical zeta function and Guillemin's trace formula}
The zeta function $ \zeta_1 $ defined in \eqref{zeta}
is closely related to the Ruelle zeta function -- see \cite{GLP},\cite{DZ}
and references given there.
The right hand side in \eqref{zeta} converges for $\Im\lambda>C_1$ and
it continues analytically to the entire plane.
 The Pollicott-Ruelle resonances are exactly the zeros of $ \zeta_1 $.
  We recall the (Atiyah--Bott--)Guillemin's trace formula \cite{Gu} (see \cite[Appendix B]{DZ} for a proof):
\begin{equation}
\label{Guillemin}
\tr^\flat e^{-itP}=\sum_{\gamma\in\mathcal{G}}\frac{T_\gamma^\#\delta(t-T_\gamma)}{|\det(I-\mathcal{P}_\gamma)|}, \ \ \   t>0.
\end{equation}
Therefore we have
\begin{equation*}
\frac{d}{d\lambda}\log\zeta_1(\lambda)=\frac{1}{i}\sum_\gamma\frac{T_\gamma^\# e^{i\lambda T_\gamma}}{|\det(I-\mathcal{P}_\gamma)|}
=\frac{1}{i}\int_0^\infty e^{it\lambda}\tr^\flat e^{-itP}dt.
\end{equation*}
From \eqref{Guillemin}, $\tr^\flat e^{-itP}=0$ on $(0,t_0)$ if $t_0<\inf\{T_\gamma:\gamma\in\mathcal{G}\}$. Formally, we can write (see \cite[\S 4]{DZ} for the
justification)
\begin{equation*}
\frac{d}{d\lambda}\log\zeta_1(\lambda)
=\frac{1}{i}\int_{t_0}^\infty e^{it\lambda}\tr^\flat e^{-itP}dt
=\tr^\flat\left(\frac{1}{i}e^{-it_0(P-\lambda)}
\int_0^\infty e^{it\lambda}e^{-itP}dt\right).
\end{equation*}
Therefore by \eqref{resolvent} we have
\begin{equation}
\label{zetaresolvent}
\frac{d}{d\lambda}\log\zeta_1(\lambda)=\tr^\flat(e^{-it_0(P-\lambda)}(P-\lambda)^{-1}).
\end{equation}
The wavefront set condition \eqref{wavefrontset} shows that
\begin{equation*}
\begin{split}
& \WF'(e^{-it_0(P-\lambda)}(P-\lambda)^{-1})\subset\\
& \ \ \ \ \ \ \{((x,\xi),(y,\eta)) \; :\;
(e^{-t_0H_p}(x,\xi),(y,\eta))\in\Delta(T^\ast X)\cup\Omega_+\cup(E_u^\ast\times E_s^\ast)\}
\end{split}
\end{equation*}
which does not intersect $\Delta(T^\ast X)$. This justifies taking the
flat trace \eqref{eq:flat}.

Therefore $\frac{d}{d\lambda} \log\zeta_1$ has a meromorphic continuation to all of $\mathbb{C}$ with simple poles and
positive integral residues. That is equivalent to having a holomorphic
continuation of $ \zeta_1 $. This strategy for proving Smale's conjecture
on the meromorphy of Ruelle zeta functions is the starting point of
our proof of the local trace formula.

\section{Estimates on flat traces}
\label{estflt}

The key step in the proof of the trace formula is the following estimate on the flat trace of the propagated resolvent.

\begin{prop}
\label{flattracees}
Let $\widetilde{P}_h(z)$ and $\widetilde{R}_h(z)$ {  be given by 
\eqref{eq:wideP} and \eqref{modifiedresolvent},} and let $t_0\in(0,\inf T_\gamma)$. Then
\begin{equation}
\label{eq:Tofz}  T(z ) := \tr^\flat(e^{-it_0h^{-1}\widetilde{P}_h(z)}\widetilde{R}_h(z)),
\end{equation}
is well defined and holomorphic in $ z $ when $ - C_0 h \leq \Im z \leq  1, $
$ |\Re z | \leq C_1 h^{\frac12} $. Moreover, in that range of $ z $,
\begin{equation}
\label{eq:flattr1}
 T ( z ) = \mathcal O_{C_0, C_1 } (h^{-2n-1}).
\end{equation}
\end{prop}

The proof is based on a quantitative study of the proof of \cite[Theorem 8.2.4]{H} and on the wave front properties established
in \cite[\S 3.3,3.4]{DZ}. The general idea is the following: the wave
front set condition shows that the trace is well defined. The analysis
based on the properties of the semiclassical wave front set shows more:
the contribution from a microlocal neighbourhood of fiber infinity is $ \mathcal O ( h^\infty ) $. The contribution away from fiber infinity can be controlled
using the norm estimate on $ \widetilde R_h ( z ) $. Since the weights
defining the $ H_{sG} $ spaces are supported near infinity, the norm
estimates are effectively $ L^2 $ estimates. 

For the proof of \eqref{eq:flattr1} we first review
the construction of the flat trace under the wave front set
condition. Suppose that $ u \in \mathcal D' ( X \times X ) $
satisfies the (classical) wave front condition
\begin{equation}
\label{eq:wfcond}   \WF ( u ) \cap N^* \Delta ( X ) = \emptyset, \ \  \Delta ( X )
= \{ ( x, x ) :  x\in X \} \subset X \times X.
\end{equation}
If $ u $ is a Schwartz kernel of an operator $ T $ then
$ \tr^\flat T :=  \langle \iota^* u , 1 \rangle $, where $ \iota :
\Delta ( X ) \hookrightarrow X \times X $. We will recall why \eqref{eq:wfcond}
allows the definition $ \iota^* u $.
For any $x_0\in X$, we can choose a neighbourhood $U$ of $x_0$ in $X$ equipped with a local coordinate patch. For simplicity, we abuse the notation and assume $x_0\in U\subset \mathbb{R}^n$. Then $\iota(x_0)=(x_0,x_0)\in U\times U\subset \mathbb{R}^n\times\mathbb{R}^n$. The conormal bundle to the diagonal is locally given by
$$N_\iota=
\{(x,x,\xi,-\xi)\in ( U\times U) \times
( \mathbb{R}^n\times \mathbb{R}^n) \}.$$
Put $ \Gamma := \WF ( u ) $ and
$ {\Gamma}_{(x,y)}=\{(\xi,\eta) : (x, y , \xi, \eta )\in\Gamma\}$. Then
\[ {\Gamma}_{(x_0,x_0)}\cap\{(\xi,-\xi):\xi\in\mathbb{R}^n,\xi\neq0\}=\emptyset.
\]
Since ${\Gamma}_{(x_0,x_0)}$ is closed, we can find a conic neighbourhood, $V$, of $ {\Gamma}_{(x_0,x_0)}$ in $\mathbb{R}^n\times\mathbb{R}^n\setminus0$ such that
$$V\cap\{(\xi,-\xi):\xi\in\mathbb{R}^n,\xi\neq0\}=\emptyset.$$
We can also find a compact neighbourhood $Y_0$ of $(x_0,x_0)$ such that $V$ is a neighbourhood of ${\Gamma}_{(x,y)}$ for every $(x,y)\in Y_0$.
Next we choose a 
neighbourhood $X_0$ of $x_0$ such that $X_0\times X_0\Subset Y_0$. Then we have for every $x\in X_0, (\xi,\eta)\in V$,
\begin{equation}
\label{eq:nonstat} {}^t\iota'(x)\cdot(\xi,\eta)=\xi+\eta\neq0.
\end{equation}

Moreover, we can choose $ V$ so that its complement, $\complement V $, is
a small conic neighbourhood of $\{(\xi,-\xi):\xi\in\mathbb{R}^n,\xi\neq0\}$. In particular
there exists a constant $C>0$ such that in ${\complement V}$, $C^{-1}|\eta|\leq|\xi|\leq C|\eta|$. We can also assume that
\begin{equation}
\label{eq:CV}
\complement V = - \complement V .
\end{equation}

Finally we choose $\psi(x)\in C^\infty(U)$ equal to 1 on $X_0$ such that $\varphi(x,y)=\psi(x)\psi(y)\in C_0^\infty(Y_0)$, then for any $\chi\in C_0^\infty(X_0)$, $u\in C^\infty(X\times X)$, we have
\begin{equation}
\label{eq:extend} \langle\iota^\ast u,\chi\rangle =\langle\iota^\ast(\varphi u),\chi\rangle
=(2\pi)^{-2n}\int\widehat{\varphi u}(\xi,\eta)I_\chi(\xi,\eta)d\xi d\eta,
\end{equation}
where
$$I_\chi(\xi,\eta)=\int \chi(x)e^{i\langle\iota(x),(\xi,\eta)\rangle}dx
=\int\chi(x)e^{ix\cdot(\xi+\eta)}dx.$$
We claim that as long as \eqref{eq:wfcond} holds the right hand
side of \eqref{eq:extend} is well defined and hence the pull back
$ \iota^* u $ is a well defined distribution.

To see this, we first notice that if $(\xi,\eta)\in V$, then
\eqref{eq:nonstat} shows that the phase is not stationary and hence,
$|I_\chi(\xi,\eta)|\leq C_{N,\chi}(1+|\xi|+|\eta|)^{-N} $, for all $N$.
On the other hand, we have
\begin{equation}
\label{eq:widehat} |\widehat{\varphi u}(\xi,\eta)|=\left|\int\psi(x)\psi(y)u(x,y)e^{-i(x\cdot\xi+y\cdot\eta)}dxdy\right|.\end{equation}
The construction of $ V $ and \eqref{eq:wfcond} imply  that if $(\xi,\eta)\not\in V$, then
$|\widehat{\varphi u}(\xi,\eta)|\leq C_N(1+|\xi|+|\eta|)^{-N}$, for all $ N$.
When $ ( \xi, \eta ) \in V $ then, there exists $M>0$ such that
$|\widehat{\varphi u}(\xi,\eta)|\leq C_N(1+|\xi|+|\eta|)^M$.
Therefore $\langle\iota^\ast u,\chi\rangle$ is well defined. Now to define $\langle\iota^\ast u,1\rangle$, we first choose a finite partition of unity $1=\sum\chi_j$ where $\chi_j$ is constructed as above for some $x_j\in X$ (playing the role of $x_0$)
and then choose the corresponding $\psi_j$'s (playing the role of $ \psi $).
This concludes our review of the proof 
that
$ \tr^\flat T = \langle \iota^* u  , 1\rangle  $ is well defined when \eqref{eq:wfcond}
holds.

All of this can be applied to $ u = K$, the
Schwartz kernel of  $e^{-it_0h^{-1}\widetilde{P}_h(z)}\widetilde{R}_h(z)$,
with quantitative bounds in terms of $ h $.
We first estimate the wave front set of
$e^{-it_0h^{-1}\widetilde{P}_h(z)}\widetilde{R}_h(z)$. For that we
need the following
\begin{lem}
\label{l:prop}
For $ t \geq 0 $,
\begin{equation*}
\begin{split}
 \WFh' ( e^{ - i t h^{-1} \widetilde P_h ( z ) } )
\cap T^\ast(X\times X)&
\subset \{ ( e^{ t_0 H_p } ( x, \xi ) , ( x, \xi ) ) : ( x ,\xi ) \in T^\ast X \},\\
\WFh' ( e^{ - i t h^{-1} \widetilde P_h ( z ) } )
\cap S^\ast(X\times X)&
\subset \kappa(\{ ( e^{ t_0 H_p } ( x, \xi ) , ( x, \xi ) ) : ( x ,\xi ) \in T^\ast X\setminus\{0\} \}).
\end{split}
\end{equation*}
\end{lem}
\begin{proof}
We first note that the inclusion is obviously true for the
$ \WFh' ( e^{ - i t P } ) $ since the operator is the pull back by
$ \varphi_{-t}^* $. Hence the statement above will follow from showing
that $ V ( t ) := e^{ i t P} e^{ - i t h^{-1} \widetilde P_h ( z ) } $ is a pseudodifferential
operator. If $ B \in \Psi^0_h $ satisfies
\[ \WFh ( B ) \cap
\cup_{ 0 \leq |t'| \leq t }\,  e^{t'H_p}  (\WFh ( Q ) ) = \emptyset , \]
then
$ B e^{ i t P} e^{ - i t h^{-1} ( h P - i Q ) }  =  B + {\mathcal O} ( h^\infty )_{
\mathcal D' \to \CI } $.

In fact, we can use Egorov's theorem (a trivial case since $ e^{ it P } =
\varphi_t ^* $) to see that
\begin{equation}
\begin{split}
 h D_t  \left( B e^{ i t P} e^{ - i t h^{-1} ( h P - i Q ) } \right) & =
i B e^{ it P } Q  e^{ - i t h^{-1} ( h P - i Q ) } \\
& = i e^{ i t P } B ( t ) Q e^{ - i t h^{-1} ( h P - iQ ) } =
\mathcal O ( h^\infty )_{ \mathcal D' \to \CI } ,
\end{split}
\end{equation}
where $ B ( t ) := e^{-  i t P } B e^{ i t P  } $ satisfies
$ \WFh ( B ( t ) ) \cap \WFh ( Q ) = \emptyset $.
By switching the sign of $ P $ and taking adjoints we see
that the we also have
$$  e^{ i t P} e^{ - i t h^{-1} ( h P - i Q ) }  B =  B + {\mathcal O} ( h^\infty )_{
\mathcal D' \to \CI } .$$
Hence it is enough to prove that, for $ \alpha $ in \eqref{eq:condfk},
$ e^{ i tP } e^{ - i h^{-1} t ( P - i Q ) } A \in \Psi_\alpha ( X ) $,
$ \alpha < {\frac12} $, for $ A \in \Psi^{\rm{comp}}( X ) $. But that
is included in \cite[Proposition { A}.3]{NZ}.
\end{proof}

\medskip
\noindent
{\bf Remark.} {  
The assumption \eqref{eq:condfk} in the construction of $ \widetilde P_h ( z ) $ and used in the proof of Lemma \ref{l:prop}
is made for convenience only as we can then cite
\cite[Proposition { A}.3]{NZ}. }

\medskip

Inclusions \eqref{eq:WFRh} and \eqref{eq:WFhR} and
Lemma \ref{l:prop} show that
\begin{equation*}
\begin{split}
& \WFh'(e^{-it_0h^{-1}\widetilde{P}_h(z)}\widetilde{R}_h(z))
\cap T^\ast(X\times X)\subset\\
& \ \ \ \ \ \ \{((x,\xi),(y,\eta)) \; :\;
(e^{-t_0H_p}(x,\xi),(y,\eta))\in\Delta(T^\ast X)\cup\Omega_+\}
\end{split}
\end{equation*}
and
\begin{equation*}
\begin{split}
& \WFh'(e^{-it_0h^{-1}\widetilde{P}_h(z)}\widetilde{R}_h(z))
\cap S^\ast(X\times X)\subset\\
& \ \ \ \ \ \ \kappa\{((x,\xi),(y,\eta)) \; :\;
(e^{-t_0H_p}(x,\xi),(y,\eta))\in\Delta(T^\ast X)\cup\Omega_+\cup(E_u^\ast\times E_s^\ast)\setminus\{0\}\}.
\end{split}
\end{equation*}
In particular, for all $0 < h < 1 $,
\begin{equation*}
\begin{split}
& \WF'(e^{-it_0h^{-1}\widetilde{P}_h(z)}\widetilde{R}_h(z))\subset\\
& \ \ \ \ \ \ \{((x,\xi),(y,\eta)) \; :\;
(e^{-t_0H_p}(x,\xi),(y,\eta))\in\Delta(T^\ast X)\cup\Omega_+\cup(E_u^\ast\times E_s^\ast)\},
\end{split}
\end{equation*}
{  satisfying \eqref{ftc},} that is, does not intersect with $\Delta(T^\ast X)$ and hence $\tr^\flat(e^{-it_0h^{-1}\widetilde{P}_h(z)}\widetilde{R}_h(z))$ is well-defined.

Using a microlocal partition of unity,
$ I=\sum_{j=1}^J B_j + \mathcal O ( h^\infty )_{ \mathcal D'
\to \CI } $, $ B_j\in\Psi^0_h(X)$ {  (see for instance \cite[Proposition E.34]{res})}
we only need to prove that
\[
\begin{split}
& {\rm(i)}  \WFh(B) \subset \neigh_{ T^* X } ( x_0, \xi_0 ) ,  \
(x_0,\xi_0)\in T^\ast X  \Rightarrow
\tr^\flat  e^{-it_0h^{-1}\widetilde{P}_h(z)}\widetilde{R}_h(z) B
= \mathcal O ( h^{-{  2 } n-1} ) , \\
& {\rm(ii)}  \WFh(B) \subset \neigh_{ \overline T^* X } ( x_0, \xi_0 ) ,  \
(x_0,\xi_0)\in S^\ast X  \Rightarrow
\tr^\flat e^{-it_0h^{-1}\widetilde{P}_h(z)}\widetilde{R}_h(z) B = \mathcal O ( h^{\infty} ) ,
\end{split} \]
In case (ii), $ W :=\neigh_{ \overline T^* X } ( x_0, \xi_0 ) $ is the image of
the closure of a conic neighbourhood of $(x_0,\xi_0)$ in $T^\ast X$, under
the map $ T ^* X \to \overline T^* X $.
{  In fact, given (i) and (ii), we can use a microlocal partition of unity to write
$$ e^{-it_0h^{-1}\widetilde{P}_h(z)}\widetilde{R}_h(z)=\sum_{j=1}^J
 e^{-it_0h^{-1}\widetilde{P}_h(z)}\widetilde{R}_h(z) B_j+O(h^\infty)_{\mathcal{D}'\to \CI}$$
where each $B_j$ satisfies either (i) or (ii) and this proves 
\eqref{eq:flattr1}.}

For each case, we repeat the construction with the Fourier transform replaced by the semiclassical Fourier transform. Let
$u=K_h$ be the Schwartz kernel
of $e^{-it_0h^{-1}\widetilde{P}_h(z)}\widetilde{R}_h(z)B$.
Then, in the notation of \eqref{eq:extend},
\begin{equation}
\label{eq:iota}
\langle\iota^\ast u,\chi\rangle =\langle\iota^\ast(\varphi u),\chi\rangle
=(2\pi h)^{-2n}\int\mathcal{F}_h(\varphi u)(\xi,\eta)I_{\chi,h}(\xi,\eta)d\xi d\eta,
\end{equation}
where now
\begin{equation}
\label{eq:Ichih} I_{\chi,h}(\xi,\eta)=\int\chi(x)e^{i\langle\iota(x),(\xi,\eta)\rangle/h}dx
=\int\chi(x)e^{ix\cdot(\xi+\eta)/h}dx. 
\end{equation}

 If $\WFh(B)$ is contained in a small compact neighbourhood $W$ of $(x_0,\xi_0)$, we can assume in the partition of unity $1=\sum\chi_j$ (see the
argument following \eqref{eq:widehat}), $\pi(W)\subset X_0$ for some coordinate patch $X_0$ and $\pi(W)\cap\supp\chi_j=\emptyset$ except for the one in this coordinate patch, say $\psi=\psi_0$. For $j\neq0$,
since
\begin{equation*}
\WFh'(\varphi_j u)\subset\WFh' ( u ) \cap[(\overline{T}^\ast X)\times\WFh(B)]
\cap[(\overline{T}^\ast\supp\psi_j )\times(\overline{T}^\ast\supp\psi_j)]=\emptyset,
\end{equation*}
we have
\begin{equation*}
 \mathcal{F}_h(\varphi_j u)(\xi,\eta)=\mathcal{O}(h^\infty(1+|\xi|+|\eta|)^{-\infty}),
\end{equation*}
and thus
$ \langle\iota^\ast u,\chi_j \rangle=\mathcal{O}(h^\infty)$.
Therefore we only need to consider the coordinate patch $X_0$ centered at $x_0$ and the corresponding $\chi,\psi$ constructed as before.
We note that  $ I_{\chi, h } (\xi, \eta ) = \mathcal O (h^\infty(1+|\xi|+|\eta|)^{-\infty})$ uniformly for $(\xi,\eta)\in V$ ($ I_{\chi, h } $ is defined in \eqref{eq:Ichih}
and again we use the notation introduced before \eqref{eq:extend}).
Hence we only need to to estimate
\begin{equation*}
\left|\int_{{\complement V}}\mathcal{F}_h(\varphi u)(\xi,\eta)I_{\chi,h}(\xi,\eta)d\xi d\eta\right|\leq\int_{{\complement V}}|\mathcal{F}_h(\varphi u)(\xi,\eta)|d\xi d\eta.
\end{equation*}
Here
\begin{equation}
\begin{split}
\mathcal{F}_h(\varphi u)(\xi,\eta)
=&\int \psi(x)\psi(y)u(x,y)e^{-i(x\cdot\xi+y\cdot\eta)/h}dxdy\\
=&\langle e^{-it_0h^{-1}\widetilde{P}_h(z)}\widetilde{R}_h(z)B(\psi(y)e^{-iy\cdot\eta/h}), \psi(x)e^{- ix\cdot\xi/h}\rangle,
\end{split}
\end{equation}
where $ \langle \bullet, \bullet \rangle $ denotes distributional pairing.
We also note that  
\[ \WFh(\psi(x)e^{ix\cdot\xi/h})=\supp\psi\times\{\xi\}, \ \ 
\WFh(\psi(y)e^{-iy\cdot\eta/h})=\supp\psi\times\{-\eta\}, \ \ 
 ( \xi , \eta ) \in \complement V .\]

In case (i), we assume 
\[ \WF_h(B)\subset W=W_1 \times W_2 \text{ where $W_1=\pi(W)\subset X_0$
and  $W_2 \subset\mathbb{R}^n$ are compact.} \]
 We make the following observation:
if $ \widetilde W_2 = \{ \xi' : \exists \, \eta' \in W_2  \ ( \xi', \eta' )
\in \complement V \} $, then either $ - \eta \notin W_2 $ or $ - \xi \in \widetilde W_2 $.
(Here we used the symmetry \eqref{eq:CV}.)
Hence if $A\in\Psi_h^{{\rm{comp}}} (X)$, $ \WFh ( I - A ) \cap
\widetilde{W}_1\times\widetilde{W}_2 = \emptyset$,
where $\widetilde{W}_1$ is a small neighbourhood of $\supp\psi$, then
\begin{equation*}
\mathcal{F}_h(\varphi u)(\xi,\eta)=
\langle e^{-it_0h^{-1}\widetilde{P}_h(z)}\widetilde{R}_h(z)B(\psi(y)e^{-iy\cdot\eta/h}), A(\psi(x)e^{- ix\cdot\xi/h})\rangle+\mathcal{O}(h^\infty).
\end{equation*}
Therefore
\begin{equation*}
\begin{split}
|\mathcal{F}_h(\varphi u)(\xi,\eta)|
=&\;|\langle e^{-it_0h^{-1}\widetilde{P}_h(z)}\widetilde{R}_h(z)B(\psi(y)e^{-iy\cdot\eta/h}), A(\psi(x)e^{- ix\cdot\xi/h})\rangle|+\mathcal{O}(h^\infty)\\
\leq &\; C\|Ae^{-it_0h^{-1}\widetilde{P}_h(z)}\widetilde{R}_h(z)B\|_{L^2\to L^2}+\mathcal{O}(h^\infty)\leq Ch^{-1}\\
\end{split}
\end{equation*}
where we use the estimate \eqref{modifiedresolvent} and the fact that microlocally on $\WFh(A)\times\WFh(B)$ which is a compact set in $T^\ast(X\times X)$, $H_{sG(h)}$ is equivalent to $L^2$ uniformly.
Combined with \eqref{eq:iota} this finishes the proof for case (i).

In case (ii), we again assume that 
$ \WFh(B)\subset W=W_1\times W_2$ where $W_1=\pi(W)\subset X_0$ is a small compact neighbourhood of $x_0$ but now $W_2\subset\bar{\mathbb{R}}^n=\mathbb{R}^n\cup\partial \, \bar {\mathbb{R}}^n$ is a small conic neighbourhood of $\xi_0\in
\partial \,\bar{ \mathbb{R}}^n$ intersecting with $\{|\xi|\geqslant C\}$. As in case (i),
we put
$ \widetilde W =\widetilde{W}_1\times\widetilde{W}_2$ such that $\widetilde{W}_1$ is a small neighbourhood of $\supp\psi$ and $\widetilde{W}_2$ is a small neighbourhood of ${\complement V}(W_2)$, which is again a small conic neighbourhood of $\xi_0$.

We then  choose $A\in\Psi_h^0(X)$ such that $ \WFh ( I - A) \cap \widetilde{W}
= \emptyset $,  and  $\WFh(A)$ is contained in a small neighbourhood of $\widetilde{W}$.
We have
\[ (\xi,\eta)\in {\complement V} \ \Longrightarrow  \ \text{ (a)  $-\eta\notin W_2$
or (b) $- \xi\in\widetilde{W}_2$.}
\]
In the case (a)  we have
\begin{equation*}
|\mathcal{F}_h(\varphi u)(\xi,\eta)|=\mathcal{O}(h^\infty(1+|\xi|+|\eta|)^{-\infty}).
\end{equation*}
In the case
 (b) we  need a uniform estimate for $\langle\xi\rangle^N|\mathcal{F}_h(\varphi u)(\xi,\eta)|$ where $N $ is large.
To do this, we use the notation from the proof of Lemma \ref{l:prop} and write
\begin{equation*}
\begin{split}
\langle\xi\rangle^N\mathcal{F}_h(\varphi u)(\xi,\eta)
=&\;\langle e^{-it_0h^{-1}\widetilde{P}_h(z)}\widetilde{R}_h(z)B(\psi(y)e^{-iy\cdot\eta/h}),\langle\xi\rangle^N\psi(x)e^{-ix\cdot\xi/h}\rangle\\
=&\;\langle \varphi_{-t_0}^\ast V (t_0) \widetilde{R}_h(z)B(\psi(y)e^{-iy\cdot\eta/h}),A(\langle\xi\rangle^N\psi(x)e^{-ix\cdot\xi/h})\rangle+\mathcal{O}(h^\infty)\\
=&\;\langle V(t_0)\widetilde{R}_h(z)B(\psi(y)e^{-iy\cdot\eta/h}),\varphi_{t_0}^\ast A(\langle\xi\rangle^N\psi(x)e^{-ix\cdot\xi/h})\rangle+\mathcal{O}(h^\infty).
\end{split}
\end{equation*}
We notice that $\WFh(\langle\xi\rangle^N\psi(x)e^{-ix\cdot\xi/h})=\supp\psi\times\{-\xi\}$, and
\begin{equation*}
\|\langle\xi\rangle^N\psi(x)e^{-ix\cdot\xi/h}\|_{H_h^{-N}}=\mathcal{O}(1)
\end{equation*}
uniformly in $\xi$. Since $ t_0 $ is small we can choose $W$ and $\widetilde{W}$ small
enough, so that $e^{-t_0 H_p}\widetilde{W}\cap\widetilde{W}=\emptyset$.
Then we choose a microlocal partition of unity,  $A_1^2+A_2^2=I +
\mathcal O ( h^\infty )_{ \mathcal D' \to \CI} $,
such that $e^{-tH_p}\widetilde{W}\subset\elh(A_1)$,
$\WFh(A_1)$ is a small neighbourhood of $e^{-tH_p} \widetilde W $ and
 $\WFh(A_2)\cap e^{-t_0H_p}(\WFh(A))=\emptyset$. We have
\begin{equation}
\label{eqa}
\begin{split}
& \langle\xi\rangle^N\mathcal{F}_h(\varphi u)(\xi,\eta)
= \langle A_1 V(t_0) \widetilde{R}_h(z)B(\psi(y)e^{-iy\cdot\eta/h}),A_1\varphi_{t_0}^\ast A(\langle\xi\rangle^N\psi(x)e^{-ix\cdot\xi/h})\rangle\\
&  \ \ \ \ \ \ \ \ \ \ + \langle A_2 V(t_0) \widetilde{R}_h(z)B(\psi(y)e^{-iy\cdot\eta/h}),A_2\varphi_{t_0}^\ast A(\langle\xi\rangle^N\psi(x)e^{- ix\cdot\xi/h})\rangle+
\mathcal \mathcal{O}(h^\infty).
\end{split}
\end{equation}

We recall the following propagation estimate \cite[Propositon 2.5]{DZ} which
is essentially the clasical result of Duistermaat--H\"ormander:
\begin{prop}
\label{ppg}
Assume that $P_0\in\Psi_h^1(X)$ with semiclassical principal symbol $p-iq\in S^1_h(X)/hS_h^0(X)$ where $p\in S^1(X;\mathbb{R})$ is independent of $h$ and $q\geqslant0$ everywhere. Assume also that $p$ is homogeneous of degree 1 in $\xi$ for $|\xi|$ large enough. Let $e^{tH_p}$ be the Hamiltonian flow of $p$ on $\overline{T}^\ast X$ and $u(h)\in\mathcal{D}'(X)$, then if $A_0,B_0,B_1\in\Psi_h^0(X)$ and for each $(x,\xi)\in\WFh(A_0)$, there exists $T\geqslant0$ with $e^{-TH_p}(x,\xi)\in\elh(B_0)$ and $e^{tH_p}(x,\xi)\in\elh(B_1)$ for $t\in[-T,0]$. Then for each $m$,
\begin{equation}
\|A_0u\|_{H^m_h(X)}\leq C\|B_0u\|_{H^m_h(X)}+Ch^{-1}\|B_1P_0u\|_{H_h^m(X)}+\mathcal{O}(h^\infty).
\end{equation}
\end{prop}

We apply the proposition to $u=\widetilde{R}_h(z)B(\psi(y)e^{-iy\cdot\eta/h})$, $P_0=\widetilde{P}_h(z)$, $A_0=A_1V(t_0)$ with $\elh(B_0)$ containing $e^{-TH_p}(\WFh(A_0))$ for $T>0$ small enough and $e^{tH_p}(x,\xi)\in\elh(B_1)$ for $t\in [-T,0]$. Furthermore, we can choose $B_1$ so that $\WFh(B_1)\cap\WFh(B)=\emptyset$.
\[ \begin{split} \|A_0u\|_{H_h^N} & \leq C\|B_0u\|_{H_h^N}+Ch^{-1}\|B_1B(\psi(y)e^{-iy\cdot\eta})\|_{H_h^N}+\mathcal{O}(h^\infty)\\
& =C\|B_0u\|_{H_h^N}+\mathcal{O}(h^\infty). \end{split} \]
However, the semiclassical wavefront set condition of $\widetilde{R}_h(z)$ shows that $\WFh(B_0)\cap\WFh(u)=\emptyset$, thus
$\|A_0u\|_{H_h^N}=\mathcal{O}(h^\infty)$
and
we know the term corresponding to $A_1$ in the sum of \eqref{eqa} is $\mathcal{O}(h^\infty)$. For the other term involving $ A_2 $, we use
$$\|A_2\varphi_{t_0} ^\ast A(\langle\xi\rangle^N\psi(x)e^{ix\cdot\xi/h})\|_{H^{P}_h }
\leq \mathcal{O}(h^\infty)\|\varphi_{t_0} ^\ast A(\langle\xi\rangle^N\psi(x)e^{ix\cdot\xi/h})\|_{H_h^{-N}}=\mathcal{O}(h^\infty),$$
for any $ P $.
This is paired with the term estimated by
$$\|A_2 V ( t_0 ) \widetilde{R}_h(z)B(\psi(y)e^{-iy\cdot\eta/h})\|_{H_h^{-P }}\leq
C\|\psi(y)e^{-iy\cdot\eta/h}\|_{H_h^P } \leq C \langle \eta \rangle^P , $$
for some $ P$. Hence
\[
\langle A_2 V(t_0) \widetilde{R}_h(z)B(\psi(y)e^{-iy\cdot\eta/h}),A_2\varphi_{t_0}^\ast A(\langle\xi\rangle^N\psi(x)e^{- ix\cdot\xi/h}) \rangle = \mathcal O ( h^\infty \langle
\eta \rangle^P ). \]
Returning to \eqref{eqa} we see that
\[ \langle \xi \rangle^N \mathcal F_h ( \varphi u ) (\xi, \eta ) =
\mathcal O ( h^\infty \langle
\eta \rangle^P ). \]
Since $ |\xi | $ is comparable $ |\eta | $ in $ \complement V $, we have
$$ \mathcal F_h ( \varphi u ) (\xi, \eta ) = \mathcal O ( h^\infty
\langle ( \xi , \eta ) \rangle^{ -N + P } ) $$ 
and that concludes the
proof of \eqref{eq:flattr1}.

\section{Proof of the trace formula}
\subsection{Sketch of the proof}
We first indicate basic ideas of the proof before we go into the details -- the
principle is quite simple but the implementation involves the use of
the results of \cite{DZ} and of some ideas from \cite{S}.

In general, a trace formula such as \eqref{localtrace} follows from the finite order of the analytic continuation of $\zeta_1(\lambda)$ in the strip $\Im\lambda\geqslant-A$, that is,
from having the following estimate valid away from small neighbourhoods of resonances:
\begin{equation}
\label{finiteorder}
\left|\frac{d}{d\lambda} \log\zeta_1(\lambda)\right|=\mathcal{O}(\langle\lambda\rangle^{2n+1}).
\end{equation}
To obtain the distributional identity \eqref{localtrace}
we take $\psi\in C_0^\infty(0,\infty)$ and compute the following integral in two different ways
\begin{equation*}
\int_{\mathbb{R}}\widehat{\psi}(\lambda)\frac{d}{d\lambda}  \log\zeta_1(\lambda)d\lambda.
\end{equation*}
On one hand, we pass the integral contour to $\mathbb{R}+iB$, where $B>C_1$ so that \eqref{zeta} converges. Since there are no resonances in the upper half plane, we have
\begin{equation*}
\begin{split}
\int_{\mathbb{R}+iB}\widehat{\psi}(\lambda)\left(\frac{1}{i}\int_0^\infty e^{it\lambda}\tr^\flat e^{-itP}dt\right)d\lambda
=&\frac{1}{i}\int_0^\infty\left(\int_{\mathbb{R}+iB}\widehat{\psi}(\lambda)e^{it\lambda}d\lambda\right)\tr^{\flat}e^{-itP}dt\\
=&\int_0^\infty\psi(t)\tr^{\flat}e^{-itP}dt.
\end{split}
\end{equation*}
Guillemin's trace formula \eqref{Guillemin} gives
\begin{equation}
\int_{\mathbb{R}}\widehat{\psi}(\lambda)\frac{d}{d\lambda}\log\zeta_1(\lambda)d\lambda=\left\langle \sum_\gamma\frac{T_\gamma^\#\delta(t-T_\gamma)}{|\det(I-\mathcal{P}_\gamma)|}, \psi \right\rangle.
\end{equation}
On the other hand, we pass the integral contour to $\mathbb{R}-iA$ and we get the contribution from the poles of $\frac{d}{d\lambda}\log\zeta_1(\lambda)$ which are exactly the Pollicott-Ruelle resonances,
\begin{equation}
\sum_{\mu\in\Res(P),\Im\mu>-A}\widehat{\psi}(\mu)=\left\langle
\sum_{\mu\in\Res(P),\Im\mu>-A}e^{-i\mu t}, \psi \right\rangle.
\end{equation}
The remainder is exactly
\begin{equation}
\langle F_A, \psi\rangle: =\int_{\mathbb{R}-iA}\widehat{\psi}(\lambda)\frac{d}{d\lambda}\log\zeta_1(\lambda)d\lambda,
\end{equation}
and we want to show that $ F_A $ can be extended to a tempered distribution
supported on $ [ 0, \infty ) $ and that it satisfies \eqref{error}. The
estimate \eqref{finiteorder} is crucial here.

To see \eqref{finiteorder}, we decompose
\begin{equation}
\label{decomposition}
\begin{split}
e^{-it_0(P-\lambda)}(P-\lambda)^{-1} & =
e^{-it_0(P-i \widetilde Q-\lambda)}(P-i \widetilde Q-\lambda)^{-1}
+[(P-\lambda)^{-1}-(P-i \widetilde Q-\lambda)^{-1}]\\
& \ \ \ \ - i\int_0^{t_0}[e^{-it(P-\lambda)}-e^{-it(P-i \widetilde Q-\lambda)}]dt,
\end{split}
\end{equation}
where $ \widetilde Q = h^{-1} Q $ for a suitably chosen $ h$ depending on the range of $ \lambda $'s. This is valid from $ \Im \lambda \gg 0 $ and then continues analytically to $ \mathbb C $ on the level of distributional Schwartz kernels.

The first term is holomorphic in $\lambda$ and can be estimated by Proposition \ref{flattracees} in the semiclassical setting.

The second term on the right hand side of \eqref{decomposition} is of trace class if $\lambda$ is not a resonance. To see this, we use the following formula
\begin{equation}
\label{res1}
\begin{split}
(P-\lambda)^{-1}-(P-i\widetilde Q-\lambda)^{-1}  & =
[(P-\lambda)^{-1}(P-i \widetilde Q-\lambda)-I](P-i\widetilde Q-\lambda)^{-1}\\
& =-(P-\lambda)^{-1}i \widetilde Q(P-i \widetilde Q-\lambda)^{-1}\\
\end{split}
\end{equation}
to get
\begin{equation}
\label{res2}
(P-\lambda)^{-1}=(P-i \widetilde Q-\lambda)^{-1}[I+i\widetilde Q(P-i\widetilde Q-\lambda)^{-1}]^{-1}.
\end{equation}
By using \eqref{res2} in  \eqref{res1} we obtain
\begin{equation}
(P-\lambda)^{-1}-(P-i\widetilde Q-\lambda)^{-1}=
-(P-i\widetilde Q-\lambda)^{-1}[I+i \widetilde Q(P-i\widetilde Q-\lambda)^{-1}]^{-1}
i\widetilde Q(P-i \widetilde
Q-\lambda)^{-1}.
\end{equation}
If we denote { $F(\lambda)=I+i\widetilde{Q}(P-i\widetilde Q-\lambda)^{-1}$}, then
\begin{equation*}
F'(\lambda)=\frac{d}{d\lambda}F(\lambda)=i\widetilde Q(P-i\widetilde Q-\lambda)^{-2}.
\end{equation*}
Moreover, $F(\lambda) - I $ and $ F' ( \lambda ) $ are operators
of finite rank. By the cyclicity of the trace, we have
\begin{equation*}
\begin{split}
\tr[(P-\lambda)^{-1}-(P-i\widetilde Q-\lambda)^{-1}]=&
-\tr[I+i\widetilde Q(P-i\widetilde Q-\lambda)^{-1}]^{-1}i\widetilde Q(P-i\widetilde Q-\lambda)^{-2}\\
=&-\tr F'(\lambda)F(\lambda)^{-1}=-\frac{d}{d\lambda}\log\det F(\lambda).
\end{split}
\end{equation*}
Therefore it can be controlled by the rank of $\widetilde Q$ and the norm of $F(\lambda)$.

The third term in \eqref{decomposition} can be handled by Duhamel's principle: if $u(t): =e^{-it(P-i\widetilde Q-\lambda)}f$, then
\begin{equation*}
\partial_t u(t)=-i(P-i\widetilde Q-\lambda)u(t), \ \ \ u(0)=f.
\end{equation*}
Rewriting the equation as
$\partial_t u(t)+i(P-\lambda)u(t)=-\widetilde Qu(t)$ ,
we get
\begin{equation*}
u(t)=e^{-it(P-\lambda)}f-\int_0^t e^{-i(t-s)(P-\lambda)}\widetilde Qu(s)ds.
\end{equation*}
Therefore
\begin{equation*}
e^{-it(P-\lambda)}-e^{-it(P-i\widetilde Q-\lambda)}=\int_0^te^{-i(t-s)(P-\lambda)}\widetilde Qe^{-is(P-i\widetilde Q-\lambda)}ds.
\end{equation*}
This shows that the left hand side is also of trace class and its trace class norm is controlled by the trace class norm of $\widetilde Q$.

{  To carry out the strategy above} we need to choose correct contours and to obtain
a local version of \eqref{finiteorder} using $ \det F ( \lambda ) $.
For that we break the infinite contour into a family of finite contours and use the semiclassical reduction to treat the zeta function on each contour separately.
That involves choices of $ h $ so that $ z = h \lambda $ is in an
appropriate range.

\subsection{The contours for integration}
In this section, we choose contours for integration. First, we decompose the region $\Omega=\{\lambda\in\mathbb{C}:-A\leq \Im\lambda\leq B\}$ into dyadic pieces: fix $ E > 0 $ and put $\Omega=\bigcup_{k\in\mathbb{Z}}\Omega_k$, where $\Omega_0=\Omega\cap\{-E\leq\Re\lambda\leq E\}$ and
\begin{gather*}
\Omega_k:=\Omega\cap\{2^{k-1}E\leq\Re\lambda\leq 2^kE\}, \ \ k>0\\
\Omega_{-k}:=\Omega\cap\{-2^kE\leq\Re\lambda\leq - 2^{k-1}E\}, \ \ k>0.
\end{gather*}
For each $k$, we write $\gamma_k=\partial\Omega_k=\bigcup_{j=1}^4\gamma^j_k$ with counterclockwise orientation.

\begin{figure}[ht]
\includegraphics[width=6.5in]{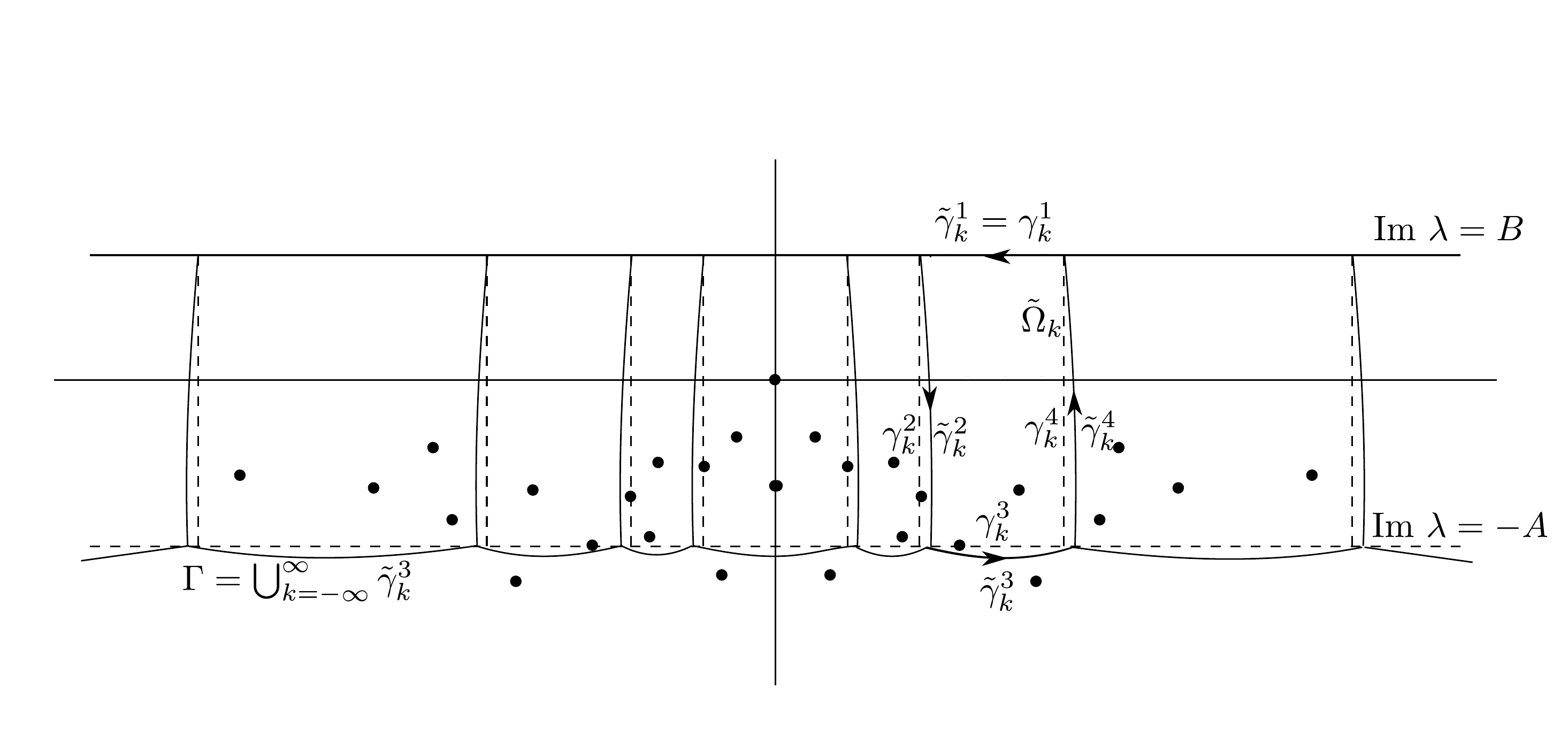}
\caption{Integration contours}
\end{figure}

Next, we shall modify $\gamma^2_k,\gamma^3_k$ and $\gamma^4_k$ to avoid the resonances. For simplicity, we only work for $k>0$ as the case for $k<0$ can be handled by symmetry. We choose $\widetilde{\gamma}^2_k,\widetilde{\gamma}^3_k$ and $\widetilde{\gamma}^4_k$ lying in
\begin{equation*}
([2^{k-1}E-1,2^kE+1]+i[-A-1,B])\setminus([2^{k-1}E+1,
2^kE-1]+i[-A,B])
\end{equation*}
so that $\widetilde{\gamma}^2_k\subset[2^{k-1}E-1,2^{k-1}E+1]
+i[-A,B]$ connects $2^{k-1}E+iB$ with a point $w_k$ which lies
on $[2^{k-1}E-1,2^{k-1}E+1]-iA$, $\widetilde{\gamma}^4_k=-\widetilde{\gamma}^2_{k+1}$; $\widetilde{\gamma}^3_k\subset[2^{k-1}E-1,2^kE+1]+i[-A-1,-A]$ connects $w_k$ with $w_{k+1}$. The region bounded by  $\widetilde{\gamma}_k :=\bigcup_{j=1}^4\widetilde{\gamma}^j_k$
 is denoted as $\widetilde{\Omega}_k$, (we write $\widetilde{\gamma}^1_k=\gamma^1_k$). Then we have { 
\begin{equation*}
\Omega\subset\widetilde{\Omega}=\bigcup_{k\in\mathbb{Z}}\widetilde{\Omega}_k
\subset\{\lambda\in\mathbb{C},-A-1\leq\Im\lambda\leq B\}
\end{equation*}}
and all $\widetilde{\Omega}_k$ have disjoint interiors.

For convenience, we turn into the semiclassical setting.
Let $W_h=h\widetilde{\Omega}_k$ where $h^{-1/2}=2^kE$, then
\begin{equation*}
\textstyle [\frac{1}{2}h^{1/2}+h,h^{1/2}-h]+i[-Ah,Bh]\subset W_h\subset
[\frac{1}{2}h^{1/2}-h,h^{1/2}+h]+i[(-A-1)h,Bh].
\end{equation*}
Moreover, $\rho_h:=\partial W_h=\bigcup_{j=1}^4\rho_h^j$ where $\rho_h^1$ is the horizontal segment $[\frac{1}{2}h^{1/2},h^{1/2}]+iBh$ with negative orientation; $\rho_h^2\subset[\frac{1}{2}h^{1/2}-h,\frac{1}{2}h^{1/2}+h]+i[-Ah,Bh]$ connects $\frac{1}{2}h^{1/2}+iBh$ with a point $z_h\in[\frac{1}{2}h^{1/2}-h,\frac{1}{2}h^{1/2}+h]-iAh$; $\rho_h^4\subset[h^{1/2}-h,h^{1/2}+h]+i[-Ah,Bh]$ connects a point $z'_h\in[h^{1/2}-h,h^{1/2}+h]-iAh$ with $h^{1/2}+iBh$;
and $\rho_h^3\subset[\frac{1}{2}h^{1/2}-h,h^{1/2}+h]$ connects $z_h$ with $z'_h$.

We have the following contour integration
\begin{equation}
\label{contour}
\oint_{\rho_h}\widehat{\psi}_h(z)\frac{d}{dz}\log\zeta_h(z)dz=
\sum_{z_j\in \Res_h(P)\cap W_h}\psi_h(z_j).
\end{equation}
Here we write $\widehat{\psi}_h(z)=\widehat{\psi}(z/h)$, $\zeta_h(z)=\zeta_1(z/h)$, $\Res_h(P)=h\Res(P)$.

We rewrite the decomposition \eqref{decomposition} in this scaling:
\begin{equation}
\label{hdecom}
\begin{split}
\frac{d}{dz}\log\zeta_h(z) & = h \tr^\flat(e^{-it_0h^{-1}P_h(z)}R_h(z))\\
& = \tr^\flat(e^{-it_0h^{-1}\widetilde{P}_h(z)}\widetilde{R}_h(z))
+ \tr(R_h(z)-\widetilde{R}_h(z))\\
& -\frac{i}{h}\tr\int_0^{t_0}[e^{-ith^{-1}P_h(z)}-e^{-ith^{-1}\widetilde{P}_h(z)}]dt.
\end{split}
\end{equation}
Then as in the discussion after \eqref{decomposition}, in the region $-C_0h\leq \Im z\leq 1, |\Re z|\leq 2h^{1/2}$, we can apply Proposition
\ref{flattracees} to obtain
\begin{equation}
\label{hes1}
\left|\tr^\flat(e^{-it_0h^{-1}\widetilde{P}_h(z)}\widetilde{R}_h(z))\right|=\mathcal{O}(h^{-2n-1}).
\end{equation}
Also we have
\begin{equation}
\label{hes2}
\left \|\int_0^{t_0}[e^{-ith^{-1}P_h(z)}-e^{-ith^{-1}\widetilde{P}_h(z)}]dt \right\|_{\rm{tr}} =\mathcal{O}(h^{-n-1}).
\end{equation}
For the second term, we have
\begin{equation*}
\tr(R_h(z)-\widetilde{R}_h(z))=-\frac{d}{dz}\log\det F(z),
\end{equation*}
where
$F(z)=I+iQ\widetilde{R}_h(z)$ is a Fredholm operator and the poles for $F(z)^{-1}$ coincides with the resonances. Moreover, by \eqref{finiterank}, \eqref{modifiedresolvent} and Weyl's inequality, we have
\begin{equation}
\label{detes1}
|\det F(z)|\leq (Ch^{-1})^{Ch^{-n}}\leq Ce^{Ch^{-n-1}}.
\end{equation}

Moreover, when $\Im z\geqslant C_1h$, we have $F(z)=I+iQ\widetilde{R}_h(z)=P_h(z)\widetilde{R}_h(z)$, so $F(z)$ is invertible and $F(z)^{-1}=\widetilde{P}_h(z)R_h(z)$. Therefore
\begin{equation*}
\begin{split}
\|F(z)^{-1}\|_{H_{sG(h)}\to H_{sG(h)}}
\leq\|\widetilde{P}_h(z)\|_{D_{sG(h)}\to H_{sG(h)}}&\|R_h(z)\|_{H_{sG(h)}\to D_{sG(h)}}\\
\leq\|\widetilde{P}_h(z)\|_{D_{sG(h)}\to H_{sG(h)}}&(\|R_h(z)\|_{H_{sG(h)}\to H_{sG(h)}}\\
&+\|hPR_h(z)\|_{H_{sG(h)}\to H_{sG(h)}})
\leq Ch^{-1}.
\end{split}
\end{equation*}
We can also write $F(z)^{-1}=I-iQR_h(z)$ which gives the estimate
\begin{equation}
\label{detes2}
|\det F(z)^{-1}|\leq (Ch^{-1})^{Ch^{-n}}\leq Ce^{Ch^{-n-1}}.
\end{equation}

We recall a lower modulus theorem due to H. Cartan { (see \cite[\S 11.3, Theorem 4]{Le})} : Suppose that $g$ is holomorphic in $D(z_0,2eR)$ and $g(z_0)=1$. Then for any $\eta>0$,
\begin{equation}
\label{lowermod}
\log|g(z)|\geqslant-\log(15e^3/\eta)\log\max_{|z-z_0|<2eR}|g(z)|,\;\;
z\in D(z_0,R)\setminus\mathcal{D},
\end{equation}
where $\mathcal{D}$ is a union of discs with the sum of radii less than $\eta R$. With the help of this lower modulus theorem, we can make a suitable choice of integration contour.

\begin{lem}
We can choose $\widetilde{\gamma}_k$ suitably such that in addition to the assumptions above, we have
\begin{equation}
\label{detbound}
|\log\det F(z)|=\mathcal{O}(h^{-n-1})
\end{equation}
when $z\in\rho_h$.
\end{lem}
\begin{proof}
We shall apply the lower modulus theorem with $z_0=\frac{1}{2}h^{1/2}+iBh/2$ and $R=C_0'h$ where $C_0'$ is large enough, so that
\begin{equation*}
\begin{split}
[\textstyle{\frac{1}{2}h^{1/2}-h,\frac{1}{2}h^{1/2}+h}]+i[(-A-1)h,Bh]   & \subset D(z_0,R) \subset D(z_0,2eR) \\
 & \subset[-2h^{1/2},2h^{1/2}]+i[-C_0h,1].
\end{split}
\end{equation*}
In addition, we let $\eta$ be small enough, so that $\eta R<h$.
Since $\widetilde{R}_h(z)$ is holomorphic in $D(z_0,2eR)$, so is $\det F(z)$. Moreover, from \eqref{detes1}, we see
\begin{equation*}
\log\max_{|z-z_0|<2eR}|\det F(z)|\leq Ch^{-n-1}.
\end{equation*}

On the other hand, at $z=z_0$, by \eqref{detes2}
\begin{equation*}
|\det F(z_0)^{-1}|\leq e^{-Ch^{-n-1}}.
\end{equation*}
Applying the lower modulus theorem with $g(z)=\det F(z)\det F(z_0)^{-1}$, we can choose $\widetilde{\gamma}^2_k$ so that
\begin{equation*}
|\log\det F(z)|=\mathcal{O}(h^{-n-1})
\end{equation*}
when $z\in h\widetilde{\gamma}^2_k$. Taking $\widetilde{\gamma}^4_k=-\widetilde{\gamma}^2_{k+1}$ as required, it is clear that we also have
\begin{equation*}
|\log\det F(z)|=\mathcal{O}(h^{-n-1})
\end{equation*}
when $z\in h\widetilde{\gamma}^4_k$.

{  Finally, we can apply the lower modulus theorem to a sequence of balls $D(z_{0j},R)$ which are translations of $D(z_0,R)$ as above by $jh$, $0\leqslant j\leqslant \frac{1}{2}h^{-1/2}$. In particular, $z_{0j}=z_0+jh$ and we have
\begin{equation*}
\begin{split}
[\textstyle{\frac{1}{2}h^{1/2}+jh-\frac{1}{2}h,\frac{1}{2}h^{1/2}+jh+\frac{1}{2}h}]+i[(-A-1)h,-Ah]   & \subset D(z_0,R) \subset D(z_0,2eR) \\
 & \subset[-2h^{1/2},2h^{1/2}]+i[-C_0h,1].
\end{split}
\end{equation*}
We can now repeat the argument above and notice that all the estimates hold uniformly in $j$ to conclude that we can choose a curve in $[\frac{1}{2}h^{1/2}+jh-\frac{1}{2}h,\frac{1}{2}h^{1/2}+jh+\frac{1}{2}h]$ for each $j$ such that $|\log\det F(z)|=O(h^{-n-1})$ uniformly in $j$ on these curves and these curves form the curve $h\tilde{\gamma}_3$ connecting $z_h$ and $z_h'$.}
\end{proof}

Now by the upper bound on the number of resonances, we have
\begin{equation}
\label{resw}
\#\Res_h(P)\cap W_h=\mathcal{O}(h^{-\frac{n}{2}-1}).
\end{equation}
From the proof of the lemma above, we can actually construct $\rho_h$ so that \eqref{detbound} holds in a neighborhood of $\rho_h$ of size $\sim h^{\frac{n}{2}+2}$. By Cauchy's inequality, we see that for $z\in\rho_h$,
\begin{equation}
\label{hes3}
\left|\frac{d}{dz}\log\det F(z)\right|=\mathcal{O}(h^{-\frac{3n}{2}-3}).
\end{equation}
Now combining \eqref{hes1}, \eqref{hes2} and \eqref{hes3}, we have the estimate for $z\in\rho_h$,
\begin{equation}
\label{zetabound}
\left|\frac{d}{dz}\log\zeta_h(z)\right|=\mathcal{O}(h^{-2n-1}),
\end{equation}
which is equivalent to \eqref{finiteorder}.

\subsection{End of the proof}
Now we finish the proof of the local trace formula by summing the contour integrals \eqref{contour}.

First, since $\widehat{\psi}(\lambda)=\mathcal{O}(\langle\lambda\rangle^{-\infty})$ as $\Re\lambda\to\pm\infty$ as well as all of its derivatives when $\Im\lambda$ is bounded, we have $\widehat{\psi}_h(z)=\mathcal{O}(h^\infty)$ as well as all its derivatives. Therefore by \eqref{zetabound},
\begin{equation*}
\int_{\widetilde{\gamma}_k^2}\widehat{\psi}(\lambda)\frac{d}{d\lambda}\log\zeta_1(\lambda) d\lambda =\int_{h\widetilde{\gamma}_k^2}
\widehat{\psi}_h(z)\frac{d}{dz}\log\zeta_h(z)dz=\mathcal{O}(h^\infty).
\end{equation*}
Moreover, both of the sums
\begin{equation*}
\sum_{k\in\mathbb{Z}}\int_{\widetilde{\gamma}_k^1}\widehat{\psi}(\lambda)
\frac{d}{d\lambda}\log\zeta_1(\lambda)d\lambda
=-\int_{\mathbb{R}+iB}\widehat{\psi}(\lambda)
\frac{d}{d\lambda}\log\zeta_1(\lambda)d\lambda
\end{equation*}
and
\begin{equation*}
\sum_{k\in\mathbb{Z}}\int_{\widetilde{\gamma}_k^3}\widehat{\psi}(\lambda)
\frac{d}{d\lambda}\log\zeta_1(\lambda)d\lambda=\int_{\Gamma}\frac{d}{d\lambda}\log\zeta_1(\lambda)d\lambda
\end{equation*}
converges absolutely. Here $\Gamma=\bigcup_{k=-\infty}^\infty\widetilde{\gamma}_k^3$.
On the other hand, by the upper bound of the resonance \eqref{resw}, the sum
\begin{equation*}
\sum_{\mu_j\in\Res(P)\cap\widetilde{\Omega}}\widehat{\psi}(\mu_j)
=\sum_{k\in\mathbb{Z}}\sum_{z_j\in\Res_h(P)\cap W_k}\widehat{\psi}_h(z_j)
\end{equation*}
also converges absolutely. Hence we have the following identity
\begin{equation*}
\int_{\mathbb{R}+iB}\widehat{\psi}(\lambda)
\frac{d}{d\lambda}\log\zeta_1(\lambda)d\lambda
=\sum_{\mu_j\in\Res(P)\cap\widetilde{\Omega}}\widehat{\psi}(\mu_j)+
\int_{\Gamma}\widehat{\psi}(\lambda)
\frac{d}{d\lambda}\log\zeta_1(\lambda)d\lambda.
\end{equation*}
In other words,
\begin{equation*}
\sum_{\mu\in\Res(P), \Im\mu>-A}\widehat{\psi}(\mu)
=\sum_{\gamma}\frac{T_\gamma^{\#}\delta(t-T_\gamma)}
{|\det(I-\mathcal{P}_\gamma)|}+\langle\psi,F_A\rangle
\end{equation*}
where
\begin{equation}
\label{errorform}
\langle\psi, F_A\rangle=-\sum_{\mu_j\in\Res(P)\cap\widetilde{\Omega},\Im\mu_j\leq-A}\widehat{\psi}(\mu_j)+\int_{\Gamma}\widehat{\psi}(\lambda)
\frac{d}{d\lambda}\log\zeta_1(\lambda)d\lambda.
\end{equation}
This proves \eqref{localtrace}.

So far, the distribution $F_A$ is only defined in $\mathcal{D}'(0,\infty)$. However, the right-hand side in \eqref{localtrace} has an obvious extension to $\mathbb{R}$ by zero on the negative half line as it is supported away from 0. By the polynomial upper bounds \eqref{upperbound} on the number of resonances in the strip $\Im\mu>-A$, the sum
\begin{equation*}
u_A(t)=\sum_{\mu\in\Res(P),\Im\mu>-A}e^{-i\mu t}
\end{equation*}
also has an extension to $\mathbb{R}$ which has support in $[0,\infty)$. We only need to show that $u_A$ is of finite order: For any $\varphi\in C_0^\infty(0,\infty)$,  $k\geqslant0$, we have
$\widehat{\varphi^{(k)}}(\lambda)=(i\lambda)^k\widehat{\varphi}(\lambda)$. Therefore we can write
\begin{equation*}
\langle u_A,\varphi\rangle=\sum_{\mu\in\Res(P),\Im\mu>-A}\widehat{\varphi}(\mu)=\sum_{\mu\in\Res(P),\Im\mu>-A}(i\mu)^{-k}\widehat{\varphi^{(k)}}(\mu)
\end{equation*}
When $k$ is large, the sum
\begin{equation*}
\sum_{\mu\in\Res(P),\Im\mu>-A}|\mu|^{-k}
\end{equation*}
converges absolutely. Therefore we have the finite order property of $u_A$. Moreover, any two such extensions of $u_A$ are only differed by a distribution $v$ supported at $\{0\}$, that is,  a linear combination of delta function and its derivatives.

Now we can certainly extend $F_A$ to a distribution on $\mathbb{R}$ with support in $[0,\infty)$. Since $\check{v}$ is a polynomial in the whole complex plane. Therefore choice of the extension of $u_A$ does not affect the estimate on $\widehat{F}_A$.

Finally, we give the {  desired} estimate on $\widehat{F}_A$. This follows from the fact $e^{\eta t}F_A\in\mathcal{S}'$ for any $\eta<A$ and \cite[Theorem 7.4.2]{H}. To see this, we only need to show that we can extend \eqref{errorform} to $\psi\in C^\infty(\mathbb{R})$ with support in $(0,\infty)$ such that $\varphi=e^{-\eta t}\psi\in\mathcal{S}$. If $\psi$ has compact support, then
\begin{equation*}
\widehat{\psi}(\lambda)=\int\psi(t)e^{-it\lambda}dt
=\int\varphi(t)e^{-it(\lambda+i\eta)}dt=\widehat{\varphi}(\lambda+i\eta).
\end{equation*}
Therefore \eqref{errorform} can be rewritten as
\begin{equation}
\label{deffa}
\langle\psi, F_A\rangle=-\sum_{\mu_j\in\Res(P)\cap\widetilde{\Omega},\Im\mu_j\leq-A}\widehat{\varphi}(\mu_j+i\eta)+\int_{\Gamma}\widehat{\varphi}(\lambda+i\eta)\frac{d}{d\lambda}\log\zeta_1(\lambda)d\lambda.
\end{equation}
Again, by the estimate \eqref{finiteorder} and the upper bound on the resonances \eqref{upperbound}, this converges as long as $\supp\varphi\subset(0,\infty)$ which gives the definition of $e^{\eta t}F_A$ in $\mathcal{S}'$. The order of $\widehat{F}_A(\lambda)$ comes from the formula \eqref{deffa} and \eqref{finiteorder}, \eqref{upperbound}.

\section{Proof of the weak lower bound on the number of resonances}
Now we prove the weak lower bound on the number of resonances. The strategy is similar to the proof in \cite{GZ}  and we proceed by contradiction. Let
\begin{equation*}
N_A(r)=\#(\Res(P)\cap\{|\mu|\leq r,\Im\mu>-A\})
\end{equation*}
and assume that
\begin{equation}
\label{contra}
N_A(r)\leq P(\delta,A)r^\delta.
\end{equation}

We fix a test function $\varphi\in C_0^\infty(\mathbb{R})$ with the following properties:
\begin{equation*}
\varphi\geqslant0, \ \ \varphi(0)>0, \ \ \supp\varphi\subset[-1,1].
\end{equation*}
Next we set $\varphi_{l,d}(t)=\varphi(l^{-1}(t-d))$ where $d>1$ and $l<1$, so that $\varphi_{l,d}\in C_0^\infty(0,\infty)$. Therefore we can apply the local trace formula to get
\begin{equation}
\label{tr1}
\sum_{\mu\in\Res(P),\Im\mu>-A}\widehat{\varphi}_{l,d}(\mu)
+\langle F_A,\varphi_{l,d}\rangle=\sum_{\gamma}\frac{T_\gamma^\#\varphi_{l,d}(T_\gamma)}{|\det(I-\mathcal{P}_\gamma)|}.
\end{equation}

First, we note that by Paley-Wiener theorem,
\begin{equation}
\label{PW}
|\widehat{\varphi}_{l,d}(\zeta)|=|l\widehat{\varphi}(l\zeta)e^{-id\zeta}|\leq
C_N l e^{(d-l)\Im\zeta}(1+|l\zeta|)^{-N},
\end{equation}
for $\Im\zeta\leq0$ and any $ N \geq 0 $.

By the assumption, we have the following estimate on the sum on the left-hand side of \eqref{tr1},
\begin{equation}
\label{es1}
\begin{split}
\left|\sum_{\mu\in\Res(P),\Im\mu>-A}\widehat{\varphi}_{l,d}(\mu)\right|\leq&\; Cl\int_0^\infty(1+l r)^{-N}dN_A(r)\\
\leq&\; Cl\int_0^\infty\frac{d}{dr}[(1+l r)^{-N}]N_A(r)dr\\
\leq&\; CP(\delta,A)l\int_0^\infty\frac{d}{dr}[(1+l r)^{-N}]r^\delta dr\leq Cl^{1-\delta}.
\end{split}
\end{equation}

The remainder term $\langle F_A,\varphi_{l,d}\rangle$ on the left-hand side of \eqref{tr1} can be rewritten as
\begin{equation*}
\langle\check{F}_A,\widehat{\varphi}_{l,d}\rangle
=\int_{\mathbb{R}}\widehat{F}_A(-\zeta)\widehat{\varphi}_{l,d}(\zeta)d\zeta.
\end{equation*}
By \eqref{error}, we can pass the contour to $\mathbb{R}+i(\epsilon-A)$ to get
\begin{equation}
\label{es2}
\begin{split}
|\langle F_A,\varphi_{l,d}\rangle|
\leq &
\int_{\mathbb{R}+i(\epsilon-A)}|\widehat{F}_A(-\zeta)|
|\widehat{\varphi}_{l,d}(\zeta)|d\zeta\\
\leq &\;
Cl e^{(d-l)(\epsilon-A)}
\int_{\mathbb{R}+i(\epsilon-A)}\langle\zeta\rangle^{2n+1}(1+l|\zeta|)^{-2n-3}d\zeta\\
\leq &\; Cl^{-2n-1}e^{(d-l)(\epsilon-A)}
\end{split}
\end{equation}
where we use \eqref{PW} with $ N =2n+3$.

{ 
On the other hand, to get a lower bound of the right-hand side of \eqref{tr1}, we fix one primitive periodic orbit $\gamma_0$ and let $d=kT_{\gamma_0}$, $k\in\mathbb{N}$. Since every term there is nonnegative, we ignore all but the term corresponding to $\gamma_d$ which is the $k$-times iterate of $\gamma_0$ and get
\begin{equation*}
\sum_{\gamma}\frac{T_\gamma^\#\varphi_{l,d}(T_\gamma)}{|\det(I-\mathcal{P}_\gamma)|}
\geq\frac{T_{\gamma_d}^\#\varphi(0)}{|\det(I-\mathcal{P}_{\gamma_d})|}=\frac{T_{\gamma_0}\varphi(0)}{|\det(I-\mathcal{P}_{\gamma_0}^k)|}.
\end{equation*}
Let $\lambda_1,\ldots,\lambda_{n-1}$ be the eigenvalues of $\mathcal{P}_{\gamma_0}$, then for some $\alpha$ depending only on $\lambda_j$'s,
\begin{equation*}
|\det(I-\mathcal{P}_{\gamma_0}^k)|=|(1-\lambda_1^k)\cdots(1-\lambda_{n-1}^k)|\leq Ce^{k\alpha}=Ce^{\theta_0d},
\end{equation*}
if $\theta_0=\alpha/T_{\gamma_0}$.} 
This gives the lower bound
\begin{equation}
\label{es3}
\sum_{\gamma}\frac{T_\gamma^\#\varphi_{l,d}(T_\gamma)}{|\det(I-\mathcal{P}_\gamma)|}\geq Ce^{-\theta_0d}.
\end{equation}


Combining \eqref{es1},\eqref{es2},\eqref{es3}, we have the following inequality
\begin{equation*}
Cl^{1-\delta}+Cl^{-2n-1}e^{(d-l)(\epsilon-A)}\geq Ce^{-\theta d}.
\end{equation*}
We first choose $l=e^{-\beta d}$, then we have
\begin{equation*}
Ce^{-\beta d(1-\delta)}+Ce^{(d-l)(\epsilon-A)+(2n+1)\beta d}\geq Ce^{-\theta_0d}.
\end{equation*}
Notice that the constants $C$'s may depend on $A$, but not on $d$. If we choose $\beta$ and $A$ large while $\epsilon$ small so that $\beta(1-\delta)>\theta_0$ and $A-\epsilon-(2n+1)\beta>\theta_0$, then we get a contradiction as $d\to\infty$. This can be achieved when $A>A_\delta$ where
\begin{equation}
\label{adelta}
A_\delta=\theta_0(1+(2n+1)(1-\delta)^{-1}).
\end{equation}
This finishes the proof of Theorem \ref{thm2}.

\medskip

\noindent
{\bf Remark.} { 
From the proof, we see that the essential gap is bounded by $A_0=\theta_0(2n+2)$, where $\theta_0$ given above only depends on the Poincar\'{e} map associated to a primitive periodic orbit $\gamma_0$. More explicitly, 
$$\theta_0=\frac{1}{T_{\gamma_0}}\sum_{\lambda\in\sigma(\mathcal{P}_{\gamma_0}):|\lambda|>1}\log|\lambda|.$$
A weaker bound not depending on the specific orbit is given by
$\theta_0\leq\theta d_u$ where $d_u=\dim E_u$ is the dimension of the unstable fiber and $\theta$ is the Lyapunov constant of the flow given in \S 2.1.}

\appendix

\newcommand{\R}{\mathbb{R}}
\newcommand{\C}{\mathbb{C}}
\newcommand{\Z}{\mathbb{Z}}
\newcommand{\N}{\mathbb{N}}
\newcommand{\Q}{\mathbb{Q}}
\newcommand{\D}{ \Omega}
\newcommand{\B}{{\mathcal B}}
\newcommand{\hen}{\mathbb{H}^{n}}
\newcommand{\T}{\mathbb{T}^d}
\renewcommand{\H}{\mathbb{H}^2}

\newcommand{\p}{\partial}
\newcommand{\lt}{{\mathcal L}}
\newcommand{\sigap}{\Sigma_A^+}
\newcommand{\siga}{\Sigma_A}
\newcommand{\U}{{\mathcal U}}
\newcommand{\M}{{\mathcal M}}
\newcommand{\half}{{\textstyle{\frac{1}{2}}}}

\section{An improvement for weakly mixing flows}
\label{weakmix}
\begin{center}
{\sc By Fr\'ed\'eric Naud}
\end{center}

An Anosov flow is called {\em weakly mixing} if
\begin{equation}
\label{eq:topmi}   \varphi_t^* f =e^{iat}f, \ \ f \in C ( X )
\ \Longrightarrow \ a = 0 , \ \ f = \rm{const}.
\end{equation}
This condition is not always satisfied for an arbitrary Anosov flow: for example suspensions of Anosov diffeomorphism by a constant
return time are not weakly mixing. On the other hand, if the flow is {\it volume} preserving, Anosov's alternative shows \cite{Anosov1} that it is {\it either} a suspension by a constant return time function {\it or} mixing for the volume measure and hence weakly mixing. 
Assuming this weakly mixing property we will obtain a more precise strip with infinitely many resonances.
The width of that strip is given in terms of a {\em topological pressure} and we start by recalling its definition.

If $G$ is a real valued H\"older continuous function,
its { topological pressure}
can be defined by the variational formula
\begin{equation}
\label{eq:press}
P(G)=\sup_{\mu} \left ( h_\mu(\varphi_1)+\int_X Gd\mu \right),
\end{equation}
 where the supremum is taken over all $\varphi_t$-invariant probability measures, and $h_{\mu}$ is the Kolgomogorov-Sinai entropy.

If in addition
\eqref{eq:topmi} holds, that is if the flow $ \varphi_t $ is
{topologically weakly mixing}, there exists an alternative way to compute the pressure using averages over closed orbits. Let $\mathcal G$ denote the set of periodic orbits of the flow and
if $\gamma \in {\mathcal G }$ let $T_\gamma$ denote its period.
Then if $P(G)>0$ we have the following asymptotic formula:
\begin{equation}
\label{eq:PaPo}
\sum_{T_\gamma \leq T}  e^{\int_\gamma G}=\frac{e^{TP(G)}}{P(G)}+o\left (e^{TP(G)} \right), \ \
\ \ T \to + \infty .
\end{equation}
 This formula dates back to Bowen \cite{Bow} -- see
\cite[Chapter 7, p.177]{PaPo}. Notice that \eqref{eq:PaPo}
implies that one can find $C>0$ such that for all $T$ large,
 \begin{equation}
\label{eq:*}   \sum_{T-1\leq T_\gamma \leq T+1}  e^{\int_\gamma G}\geq C e^{T P(G)}. \end{equation}
The case $P(G)\leq 0$ (which will be considered here)
can be dealt with by applying the
above lower bound to
$G_\epsilon=G+(1+\epsilon) \vert P(G) \vert$ for a fixed $\epsilon>0$.

The function $ G $ used in the estimate of the strip is the
Sinai-Ruelle-Bowen (SRB) potential $ \psi^u ( x ) $,  defined as follows:
\begin{equation}
\label{eq:psiu}
\psi^{u}(x)=-\frac{d}{dt} \left \{ \log \vert \det D_x \varphi_t\vert_{E_u(x)} \vert  \right \} \vert_{t=0},
\end{equation}
where $E_u(x)$ is the unstable subspace in $T_xX$ -- \S 2.1.
 The potential $\psi^u$ is H\"older continuous on $X$ and
the associated invariant measure (equilibrium state) $\mu_u$
generalizes the Liouville measure for flows:
in general it is not absolutely continous with respect to the Lebesgue
measure but it is absolutely continuous on unstable leaves.

We can now state the result giving a gap in terms of pressure:

\begin{thm}
\label{thm3}
Suppose that $ X $ is a compact manifold and $ \varphi_t : X \to X $
is a {\em weakly mixing Anosov} flow. Let $ \psi^ u $ given
by \eqref{eq:psiu} and $ P ( 2 \psi^u ) $ be its pressure defined
by \eqref{eq:press}.

Then for any $ \epsilon > 0 $ we have
\begin{equation}
\label{eq:Naud}
\#(\Res(P)\cap\{\mu \in \CC \, : \,  \Im\mu > (2n 
+ \textstyle{\frac32}) P(2 \psi^u ) - \epsilon \})
= \infty ,  \end{equation}
where $ n = \dim X $.
\end{thm}

\medskip

 \noindent
{\bf Remark:} Since we are concerned with Anosov flows, we have automatically $P(\psi^u)=0$. This implies, using the variational formula, that we have the bounds
 $$ \sup_{\gamma\in {\mathcal G}} \left \{ -2 T_\gamma^{-1} \log \vert  \det(D_{x_\gamma}\varphi_{T_\gamma} \vert_{E_u(x)}) \vert \right \} \leq P(2\psi^u)\leq \sup_{\mu} \int_X \psi^u d\mu <0.$$
 In particular this implies that if there exists a periodic orbit $\gamma$ along which the unstable jacobian
 $$\vert  \det(D_{x_\gamma}\varphi_{T_\gamma} \vert_{E_u(x)}) \vert, $$
 is close to $1$, then the width of the strip with infinitely many Ruelle-Pollicott resonances is also small. This observation suggests 
that on the verge of non-uniformly hyperbolic dynamics,
 Ruelle-Pollicott resonances  converge to the real axis.

\medskip

\begin{proof}
Choose $ \psi \in \CIc ( ( - 2, 2 ) ) $ so that $ \psi \geq 0 $
and that $ \psi( s ) = 1 $ for $ |s | \leq 1 $. For $ t \geq 0 $
and $ \xi \in \RR $ we put
\[ \psi_{t, \xi } ( s ) := e^{  i s \xi } \psi ( s - t ) . \]
We use this function, with $ t \gg 1 $ as a test function for the 
right hand side of \eqref{localtrace} and we define
\begin{equation}
\label{eq:Stxi}  S ({ t, \xi }) := \sum_{\gamma \in \mathcal G } 
\frac{ e^{ i T_\gamma \xi} 
  \psi ( T_\gamma - t ) T^{\#}_\gamma }{ |\det( I - \mathcal P_\gamma )|} . 
\end{equation}
For $ A > 0 $, the local trace formula \eqref{localtrace} now gives
\begin{equation}
\label{eq:aplo} S ({t , \xi }) = 
\sum_{\mu\in\Res(P), \Im\mu>-A} \widehat \psi_{ t, \xi} ( \mu ) 
+ \langle F_A, \psi_{t , \xi }\rangle  . 
\end{equation}
We note that $ \widehat \psi_{ t , \xi } ( \lambda )  = e^{ i t ( \xi - \lambda ) } \widehat \psi ( \lambda - \xi ) $ is an entire function of $ \lambda $
which satisfies the estimate
\begin{equation}
\label{eq:estpsi} 
| \widehat \psi_{ t , \xi } ( \lambda ) | \leq C_N 
e^{ t \Im \lambda + 2 | \Im \lambda | } ( 1 + | \Re \lambda - \xi | )^{-N} ), 
\end{equation}
for any $ N \geq 0 $. Using \eqref{error} this implies that
\begin{equation}
\label{eq:FAp}
\langle F_A, \psi_{ t, \xi } \rangle \leq 
C_A  e^{ - ( A - \epsilon ) t } ( 1 + |\xi|)^{ 2n + 1 } .
\end{equation}

We now assume that for some $ A > 0 $ we have 
\[ \#(\Res(P)\cap\{\mu \in \CC \, : \,  \Im\mu > - A  \}) < \infty .\]
Then \eqref{eq:aplo}, \eqref{eq:FAp} and \eqref{eq:estpsi} with $N=1$ (recall that the sum over resonances is finite) show that
\begin{equation}
\label{eq:FAAp}
| S( t , \xi ) | \leq \frac{C_1}{1+\vert \xi \vert} + C_2  e^{ - ( A - \epsilon ) t } ( 1 + |\xi|)^{ 2n + 1 } .
\end{equation}

We will now average $ |S ( t, \xi ) |^2 $ against a Gaussian weight:
\[ G ( t , \sigma ) := \sigma^{\frac12} \int_{\RR} | S ( t , \xi)|^2 
e^{ - \sigma \xi^2/2 } d \xi . \]
From \eqref{eq:FAAp} and using the crude bound $(a+b)^2\leq 2(a^2+b^2)$, we see that for $ 0 < \sigma < 1 $, 
\begin{equation}
\label{eq:Gts}
 G ( t, \sigma ) \leq C_1'\sigma^{\frac12}  + C_2' \sigma^{-(2n+1)} e^{ - 2 ( A - 
\epsilon ) t } .
\end{equation}
On the other hand the definition \eqref{eq:Stxi} gives
\begin{equation*}
\begin{split}
G ( t, \sigma ) & = \sigma^{\frac12} \int_\RR \sum_{\gamma \in \mathcal G}
\sum_{\gamma' \in \mathcal G} 
\frac{ T_\gamma^{\#} T_{\gamma'}^{\#} e^{  i \xi ( T_\gamma - T_{\gamma'})}
\psi ( T_\gamma - t ) \psi ( T_{\gamma'}  - t ) }
{ | \det ( I - \mathcal P_\gamma ) | | \det ( I - \mathcal P_{\gamma'} ) |}
e^{ - \sigma \xi^2/2 } d \xi \\
& = \sqrt { 2 \pi } 
\sum_{\gamma \in \mathcal G}
\sum_{\gamma' \in \mathcal G} 
\frac{ T_\gamma^{\#} T_{\gamma'}^{\#} e^{  - ( T_\gamma - T_{\gamma'})^2 /2 \sigma }
\psi ( T_\gamma - t ) \psi ( T_{\gamma'}  - t ) }
{ | \det ( I - \mathcal P_\gamma ) | | \det ( I - \mathcal P_{\gamma'} ) |}
\end{split}
\end{equation*}
Since all the terms in the sums are non-negative we estimate 
$ G ( t, \sigma ) $ from below using the diagonal contributions only:
\begin{equation}
\label{eq:Gauss}
G ( \sigma, t ) \geq \sqrt{ 2 \pi}
\sum_{\gamma \in \mathcal G}
\frac{ (T_\gamma^{\#})^2 
\psi ( T_\gamma - t )^2} 
{ | \det ( I - \mathcal P_\gamma ) |^2} \geq
c \sum_{ t - 1 \leq T_\gamma \leq t + 1} { | \det ( I - \mathcal P_\gamma ) |^{-2}} .
\end{equation}
We now want to relate the sum on the right hand side 
 to a quantity which can be estimate using  \eqref{eq:*} with 
$ G = 2 \psi^u(x)$ defined by in \eqref{eq:psiu}.

The potential $ \psi^u $  satisfies 
the elementary formula that 
follows from the cocycle property: 
 $$\int_\gamma \psi^u=\int_0^{T_\gamma} \psi^{u}(\varphi_t x)dt=-\log \vert \det D_x \varphi_{T_\gamma} \vert_{E_u(x)} \vert, \ \ x \in \gamma. $$
see for example \cite{BowRu}.
Using \eqref{eq:*} we therefore have the lower bound
$$ \sum_{T-1\leq T_\gamma \leq T+1} e^{2\int_\gamma \psi^{u}}= \sum_{T-1\leq T_\gamma \leq T+1}
 \frac{1}{\vert \det(D_x \varphi_{T_\gamma} \vert_{E_u(x)})\vert^2}\geq C e^{T P(2 \psi^u)}.$$
Now we are almost done if we can relate $\det(I-\mathcal P\gamma)$ to $\det(D_x \varphi_{T_\gamma} \vert_{E_u(x)})$. By choosing an appropriate basis of $T_xM$ and using the hyperbolic splitting of the tangent space, it is easy to check that
\[ \begin{split} \det(I-\mathcal P\gamma) & =\det(I-D_x \varphi_{T_\gamma} \vert_{E_s(x) \bigoplus E_u(x)} )\\
& =\det(I-D_x \varphi_{T_\gamma} \vert_{E_s(x)})\det(I-D_x \varphi_{T_\gamma} \vert_{E_u(x)}). \end{split} \]
Now observe that
$[D_x \varphi_{T} \vert_{E_u(x)}) ]^{-1}=D_x \varphi_{-T}\vert_{E_u(x)}$,
so that we can write
 $$ \det(I-D_x \varphi_{T_\gamma} \vert_{E_u(x)})=\det(D_x \varphi_{T_\gamma} \vert_{E_u(x)} )\det( D_x \varphi_{-T_\gamma}\vert_{E_u(x)}-I).$$
 Using the Anosov property reviewed in \S 2.1, we have independently on the choice of $x\in \gamma$,
 $$\Vert D_x \varphi_{T_\gamma} \vert_{E_s(x)} \Vert\leq C e^{-\theta T_\gamma}\ \mathrm{and}\  \Vert D_x \varphi_{-T_\gamma} \Vert_{E_u(x)} \Vert\leq C e^{-\theta T_\gamma }.$$
We deduce therefore that there exists a large constant $C'>0$ such that for all $T$ large enough, for all periodic orbit $\gamma$ with
$T-1\leq T_\gamma \leq T+1$, 
  we have indeed
$$ \vert \det(I-D_x \varphi_{T_\gamma} \vert_{E_u(x) \bigoplus E_s(x)} ) \vert \leq C' \vert \det(D_x \varphi_{T_\gamma} \vert_{E_u(x)})\vert,$$
and
\[ G ( \sigma, t ) \geq c \sum_{ t - 1 \leq T_\gamma \leq t + 1} 
| \det ( I - \mathcal P_\gamma )|^{-2} \geq c' e^{ P ( 2 \psi_u ) t }.
\] 
Combining this with \eqref{eq:Gts} we obtain for $ t \gg 1 $ and
$ 0 < \sigma < 1$, 
\[ c_1 \sigma^{\frac12} + c_2 \sigma^{-(2n+1)} e^{ - ( A - \epsilon ) t } 
\geq e^{  P ( 2 \psi^u )t  }. \]
Now take $ \sigma = c_1^{-2} e^{  2 ( P ( 2 \psi^u ) - \epsilon ) t } $. 
Then 
\[ e^{ (P ( 2 \psi^u ) - \epsilon ) t } + c_2 '
e^{ ( - ( 2n + 1)  
2 ( P ( 2 \psi^u ) - \epsilon ) -  2 A + 2 \epsilon  )t } 
\geq e^{  P (2 \psi^u ) t } . \]
This gives a contradiction for all $ A > ( 2 n + \frac32 ) P ( 2 \psi^u ) $, 
concluding the proof.
\end{proof}

{ 
\section{Computing the power spectrum of Anosov suspension flows}
\label{suspe}

\begin{center}
{\sc By Fr\'ed\'eric Naud}
\end{center}

Let $\T:=\R^d/ \Z^d$ denote the  $d$-dimensional torus and let 
$F:\T \rightarrow \T$ be a smooth {\it Anosov diffeomorphism}. Fix $r>0$ and consider the product
$\widetilde{M}:=\T\times [0,r]$.
If one identifies each boundary components of $\widetilde {M}$ through the rule $(x,r)\sim (F(x),0)$, then 
the quotient, $M$, is a smooth manifold.
The vertical flow $$\varphi_t(x,u):=(x,u+t), $$
with the above identification is then well defined and is an Anosov flow (called the constant time suspension of $F$). The purpose of this section is to compute the Ruelle resonances of this particular family of Anosov flows and show that they form  a {\it lattice} which is related to ``the suspension time" $r$ and the Ruelle spectrum of the Anosov map $F$, see \eqref{eq:resP} for the exact formula.

Although the results are not surprising they do not seem to be available in the literature in the form presented here. 
Calculations involving the Ruelle zeta function for locally constant Axiom A flows can be found in \cite{Poll,Ruelle1} and what we show below is definitely close to the ideas in those papers. For simplicity, we will assume that $F$ is a {\it small real analytic perturbation} of a linear Anosov diffeomorphism
$A:\T \rightarrow \T$ such that we have (for $\epsilon>0$ small enough),
$$F(x)=Ax+\epsilon \Psi(x)\ \mathrm{mod}\ \Z^d,$$
where $\Psi:\T\rightarrow \R^d$ is a real-analytic map. A popular example of linear Anosov map $A$ is the celebrated {\it Arnold's cat map} induced by the $SL_2(\Z)$ element
$$A=\left ( \begin{array}{cc} 1&1\\1&2   \end{array} \right).$$

From the work of Faure-Roy \cite{Faureroy}, we know that there exists a Hilbert space $\mathcal{H}$ of hyperfunctions such that the transfer operator $T:\mathcal{H}\rightarrow \mathcal{H}$ defined by 
$$T(f):=f\circ F$$
is of trace class. In addition, we have the trace formula valid for all $p\geq 1$, 
\begin{equation}
\label{tr01}
\tr T^p=\sum_{F^p x=x} \frac{1}{\vert \det(I-D_x (F^p)) \vert} = \sum_{ k=0}^{\infty } e^{ - i p \lambda_k} , 
\end{equation}
where $ e^{ -i \lambda_j } $
are the Ruelle resonances of $F$, ordered by decreasing modulus, 
$ 1=\vert e^{- i \lambda_0} \vert \geq \ldots \geq \vert e^{-i \lambda_k} \vert\geq \ldots $, $ k \in [ 1 , K ) \cap \N $, where $ 2 \leq K \leq \infty $.
The exponent $ \lambda_k $ defined modulo $ 2 \pi  \Z $. 
 We also recall from \cite{Faureroy} that 
\begin{equation}
\label{eq:FRbound} \Im  \lambda_k \leq -C k^{1/d}, \end{equation}
and that the exponent $1/d$ is believed to be optimal.
In what follows we will only use the trace formulas \eqref{tr01}.

We now proceed to compute the Ruelle resonances of the flow by evaluating the dynamical side of the local trace formula. We use the same notations as before. Let $\varphi \in C_0^\infty( (0,\infty))$ be a test function. Using the fact that there is a one-to-one correspondence between primitive periodic
orbits of the flow $\varphi_t$ and periodic orbits of the map $F$ of the form $\{ x,Fx,\ldots, F^k x\}$ where $k$ is the least period, we can write
\[ \begin{split} \sum_{\gamma \in \mathcal{G}} \frac{T_\gamma^\# \varphi(T_\gamma)}{ \vert \det(I-\mathcal{P}_\gamma)\vert} & = 
\sum_{n\geq 1}\sum_{k\geq 1} \frac{1}{k} \sum_{T^kx=x \atop k\  \mathrm{least}} 
\frac{kr\varphi(knr)}{ \vert \det(I-D_xF^{kn})\vert} 
=r\sum_{p=1}^{\infty} \sum_{F^p x=x}  \frac{\varphi(pr)}{ \vert \det(I-D_xF^{p})\vert}. \end{split} \]
Using the trace formula (\ref{tr01}) and exchanging summation order we get
$$\sum_{\gamma \in \mathcal{G}} \frac{T_\gamma^\# \varphi(T_\gamma)}{ \vert \det(I-\mathcal{P}_\gamma)\vert}=
r\sum_{k=0}^{\infty} \sum_{p=1}^{\infty} \varphi(pr) e^{-i p \lambda_k}. $$

Using the Poisson summation formula we conclude that 
\begin{equation}
\label{eq:poisson} \sum_{\gamma \in \mathcal{G}} \frac{T_\gamma^\# \varphi(T_\gamma)}{ \vert \det(I-\mathcal{P}_\gamma)\vert}=
\sum_{k=0}^{\infty}  \sum_{j \in \Z} \widehat{\varphi}\left(\frac{2\pi}{r} j+  \frac{\lambda_k}{r} \right), \ \ \ \varphi \in \CIc ( ( 0 , \infty )  .
\end{equation}
In other words, in the sense of distributions we have a global trace
formula, 
\begin{equation}
\label{eq:globalsusp} \sum_{\gamma \in \mathcal{G}} \frac{T_\gamma^\#  \delta ( t -T_\gamma)}{ \vert \det(I-\mathcal{P}_\gamma)\vert}=
\sum_{k=0}^{\infty}  \sum_{j \in \Z} e^{ - i ( 2 \pi j + \lambda_k )t/r } , \ \ t > 0 
. 
\end{equation}
Theorem \ref{thm1} shows that for all $A>0$ this is equal to
$$\sum_{\mu \in \mathrm{Res}(P)\atop \Im(\mu)>-A} e^{ - i \mu t } + F_A (t )  , $$
where $F_A$ is a distribution supported on $[ 0 , \infty ) $ whose Fourier transform is analytic on $\{ \Im z <A\}$. Comparison with \eqref{eq:globalsusp} shows that that 
\begin{equation}
\label{eq:resP} \mathrm{Res(P)}= \left\{ ({2\pi} j + {\lambda_k})/r
\; :\; (k,j)\in ([1, K ) \cap \N ) \times \Z \right \}.\end{equation}
In the special case of {\it linear } Anosov maps there is only 
one Ruelle resonance at $ 1 $. (This can be seen immediately from the fact that
$\# \{ x\in \T  : A^nx=x \}=\vert \det(I-A^n)\vert$ so that for all $n\geq 1$ the trace formula (\ref{tr01}) gives 
$ \mathrm{Tr}(T^n)=\sum_{A^n x= x} |\det(I-A^n)|^{-1} =1$.) Therefore the resonances
of the suspension flow are given by $ (2 \pi /r) \Z $.
The fact that there are infinitely many real resonances is due to the absence of mixing of the flow (the suspension time is constant).
It is believed (but not proved so far) that typical Anosov map should have infinitely many eigenvalues $e^{-i \lambda_k}$ and therefore there should be, in general, infinitely many horizontal lines of resonances.

}

\def\arXiv#1{\href{http://arxiv.org/abs/#1}{arXiv:#1}}

\end{document}